\documentclass[11pt,oneside,reqno]{amsart}
\usepackage{amsmath,amsfonts,amssymb}
\usepackage{esint}
\usepackage{amsthm} 
\usepackage{tikz}
\usepackage{dsfont} % pour les indicatrices (\mathds{1})

\usepackage{enumitem}

\usepackage{xfrac}
\usetikzlibrary{positioning}
\usepackage{mathtools}

\usepackage{hyperref} %Leave it as the but one  last package
\usepackage[capitalise, nameinlink, noabbrev]{cleveref} %Leave it as the last package
\hypersetup{colorlinks=true, 
linkcolor=blue,
citecolor=red,
}

\theoremstyle{plain} 

\newtheorem{theorem}{Theorem}[section]
\newtheorem{proposition}[theorem]{Proposition}
\newtheorem{corollary}[theorem]{Corollary}
\newtheorem{lemma}[theorem]{Lemma}
\newtheorem{definition}[theorem]{Definition}
\theoremstyle{remark}
\newtheorem{remark}[theorem]{Remark}

\newcommand{\F}{\mathbb{F}}
\newcommand{\R}{\mathbb{R}}

\newcommand{\N}{\mathbb{N}}

\newcommand{\eps}{\varepsilon}
\newcommand{\de}{\partial}

\newcommand{\graph}{{\rm graph}}

\newcommand{\cH}{\mathcal{H}}
\newcommand{\cP}{\mathcal{P}}
\newcommand{\cA}{\mathcal{A}}
\newcommand{\cB}{\mathcal{B}}
\newcommand{\cC}{\mathcal{C}}
\newcommand{\cM}{\mathcal{M}}
\newcommand{\cN}{\mathcal{N}}

\newcommand{\cL}{\mathcal{L}}
\newcommand{\cQ}{\mathcal{Q}}
\newcommand{\cV}{\mathcal{V}}
\newcommand{\cI}{\mathcal{I}}

\newcommand{\bM}{\mathbf{M}}

\def\XXint#1#2#3{{\setbox0=\hbox{$#1{#2#3}{\int}$}
     \vcenter{\hbox{$#2#3$}}\kern-.5\wd0}}

\numberwithin{equation}{section}

\title[Quantitative isoperimetry on manifolds]{The Riemannian quantitative isoperimetric inequality}

\author{Otis Chodosh}
\address{OC: Department of Mathematics, Princeton University, Fine Hall, Washington Road, Princeton,
NJ, 08544, USA}
\address{School of Mathematics, Institute for Advanced Study, Princeton, NJ 08540, USA}
\email{ochodosh@math.princeton.edu}
\author{Max Engelstein}
\address{ME: Department of Mathematics, Massachusetts Institute of Technology, Cambridge, MA,
02139, USA.}
\email{maxe@mit.edu}
\author{Luca Spolaor}
\address{LS: Massachusetts Institute of Technology (MIT), 77 Massachusetts Avenue, Cambridge MA 02139, USA}
\email{lspolaor@mit.edu}
\begin{document}

%\tableofcontents
\begin{abstract}

We study the Riemannian quantiative isoperimetric inequality. We show that direct analogue of the Euclidean quantitative isoperimetric inequality is---in general---false on a closed Riemannian manifold. In spite of this, we show that the inequality is true generically. Moreover, we show that a modified (but sharp) version of the quantitative isoperimetric inequality holds for a real analytic metric, using the \L ojasiewicz--Simon inequality. A main novelty of our work is that in all our results we do not require any a priori knowledge on the structure/shape of the minimizers.
\end{abstract}
\maketitle

\section{Introduction}

The isoperimetric inequality on $\R^{n}$ states that $\cP(\Omega) \geq \cP(B)$ for any Caccioppoli set $\Omega$ with $|\Omega| = |B|$ with equality only for $\Omega=B$ (up to a set of measure zero). That is, the isoperimetric inequality states that balls in Euclidean space have the least perimeter for their enclosed volume. Starting with Bonnesen (cf.\ \cite{Oss:iso}), there has been considerable activity concerning quantitative versions of the isoperimetric inequality, finding geometric conditions on a set $\Omega$ that nearly achieves equality in the isoperimetric inequality. 

Recently, a (Euclidean) quantitative isoperimetric inequality holding in all dimensions has been established by Fusco, Maggi, and Pratelli, \cite{FuMaPr}. They proved that if $\Omega$ is a Caccioppoli set with $|\Omega| = |B_1(0)|$, then 
\begin{equation}\label{eq:quant-iso-Rn}
\left(\inf_{B = B_{1}(x) \subset \R^{n}} |\Omega\Delta B| \right)^{2} \leq C(n) (\cP(\Omega) - \cP(B)).
\end{equation}
By considering $C^{2}$-perturbations of the ball, the exponent on the left hand side is seen to be sharp, see \cite{Hall}. Subsequently, Figalli, Maggi, and Pratelli, \cite{FiMaPr}, and Cicalese and Leonardi, \cite{CiLe}, gave substantially different proofs of \eqref{eq:quant-iso-Rn}. See also \cite{CiLe:const,FuJu}, and \cite{CaFi,AcMoFu} for applications of such quantitative inequalities. 

In this work, we consider the analogue of \eqref{eq:quant-iso-Rn} on a closed Riemannian manifold. The symmetrization and optimal transport techniques of \cite{FuMaPr,FiMaPr} are not applicable even for non-quantitative isoperimetric inequalities on a general Riemannian manifold (see Section \ref{s:previous} below for more discussion), so we follow the selection principle approach of \cite{CiLe}. The general idea of the selection principle (which has roots in the work of White, \cite{Wh2}) is that by considering a ``worst case scenario'' for \eqref{eq:quant-iso-Rn}, one can reduce to the case where $\partial\Omega$ is a small $C^{1,\alpha}$ graph over $\partial B$. At this point work of Fuglede, \cite{Fug}, applies (in $\R^{n}$) to show that \eqref{eq:quant-iso-Rn} holds in the worst case scenario (and thus in all situations). 

In a closed Riemannian manifold $(M,g)$, it is well known that isoperimetric regions exist for all volumes $V \in (0,|M|_{g})$. However, there are surprisingly few manifolds where the explicit isoperimetric regions are known (see \cite[Appendix H]{EM:all-dim} for a recent survey). As such, the methods used in \cite{Fug} that rely explicitly on the geometry of $B \subset \R^{n}$ cannot be directly extended to a general manifold. Indeed, an estimate of the form \eqref{eq:quant-iso-Rn} is \emph{false} in a general Riemannian manifold, even for sets which are small graphs over isoperimetric regions! (We note, however, that the quantitative isoperimetric inequality in the form \eqref{eq:quant-iso-Rn} for regions in space-forms does hold \cite{BoDuzSch,BoVeDuzFus}).

Indeed, we construct the following example in \cref{ss:example}:
\begin{theorem}\label{thm:example}
For all $n\geq 2$ there exists a closed manifold $M^{n}$ with a real analytic Riemannian metric $g$, a uniquely isoperimetric region $\Omega\subset M$ and sets with smooth boundary, $E_{k}$, so that $|\Omega\Delta E_{k}|_{g}\to 0$ but the sets $E_{k}$ do not satisfy the analogue of \eqref{eq:quant-iso-Rn}, i.e.
\[
\frac{(|\Omega\Delta E_{k}|_{g})^{2}}{\cP^{g}(E_{k}) - \cP^{g}(\Omega)} \to \infty. 
\]
In fact, for any $\gamma>0$ fixed, there exists a real analytic $g$, depending on $\gamma$, so that
\[
\frac{(|\Omega\Delta E_{k}|_{g})^{2+\gamma}}{\cP^{g}(E_{k}) - \cP^{g}(\Omega)} \to \infty. 
\]
Finally, we show that there is a $g$ smooth but not real analytic on $M^{n}$ so that one cannot bound $\cP^{g}(E_{k})-\cP^{g}(E)$ by any power of $|\Omega\Delta E_{k}|_{g}$, i.e.,
\[
\frac{(|\Omega\Delta E_{k}|_{g})^{2+\gamma}}{\cP^{g}(E_{k}) - \cP^{g}(\Omega)} \to \infty, \quad \text{for all}\;\; \gamma>0. 
\]

\end{theorem}

As such, the natural analogue of \eqref{eq:quant-iso-Rn} cannot hold in general. Nevertheless, we prove that \eqref{eq:quant-iso-Rn} holds generically in the following sense. Let $\Gamma$ denote the the set of $C^{3}$-metrics on a given Riemannian manifold. 
\begin{theorem}\label{thm:generic-isop}
Let $M^n$ be a closed manifold, $2\leq n\leq 7$. There exists an open and dense subset $\mathcal G\subset \Gamma$ with the following property. If $g\in \mathcal G$, then there exists an open dense subset $\cV \subset (0,|M|_{g})$ so that for $V_{0} \in \cV$, there is $C=C(g,V_{0}) > 0$ so that 
\begin{equation}\label{eq:generic-quant-iso}
\cP^{g}(E) - \cI^{g}(V_{0}) \geq C \alpha_{g}(E)^{2}
\end{equation}
for any $E\subset M$ with $|E|_{g} = V_{0}$. Here,
\[
\cI_{g}(V_{0}) : = \inf\{ \cP^{g}(\Sigma) : |\Sigma|_{g} = V_{0}\}
\]
is the isoperimetric profile and the manifold Fraenkel asymmetry is
\[
\alpha_{g}(E) : = \inf\{|E \Delta \Sigma|_{g} : \Sigma \in \cM^{g}_{V_{0}}\},
\]
for $\cM^{g}_{V_{0}}$ the set of $\Sigma$ attaining the infimum in $\cI_{g}(V_{0})$. 
\end{theorem} 

A key element of the proof here is a bumpyness result in the spirit of \cite{White_b1,White_b2}, except with a volume constraint. 

\begin{remark}
We note that given a metric $g$, fixing $V \in (0,|M|_{g})$ one can always find a nearby $g$ so that \eqref{eq:generic-quant-iso} holds for volume $V$ (without needing to perturb $V$). See \cref{coro:fixV-perturbg}. 
\end{remark}

Moreover, for any real analytic metric $g$, we prove an analogue of \eqref{eq:quant-iso-Rn} that holds for all volumes. 
\begin{theorem}\label{thm:quantit_analytic}
For $2\leq n\leq 7$, assume that $(M^n,g)$ is a real analytic, closed Riemannian manifold, and let $0<V_{0}<|M|_g$. There exist constants $C_0>0, \gamma \geq 0$, depending only on $(M,g)$ and $V_{0}$, so that
\begin{equation}\label{e:quant}
\cP^{g}(E) - \cI^{g}(V_{0}) \geq C_0 \alpha_{g}(E)^{2+\gamma}.
\end{equation}
for any $E \subset M$ with $|E|_{g} = V_{0}$. 
% that if $\Sigma\in \cA^g_{V_0}$ is a minimizer of \eqref{e:isop}, then
%	\begin{equation}\label{e:quant}
%	\delta\cP(E,g) \geq C_0 \,\left(\alpha(E,g)\right)^{2+\gamma}\,\qquad \mbox{for every $E\in \cA^g_{V_0}$}\,.
%	\end{equation}
\end{theorem}

As remarked above, the main difference between \cref{thm:quantit_analytic} and essentially all the known quantitative inequalities is that we don't have any {\it a priori} knowledge of the structure/shape or any classification on the minimizers of \eqref{e:isop}. For this reason we expect this method to be applicable to a variety of other problem. On the other hand, the price that we have to pay is the exponent $\gamma>0$ (see \cref{s:previous} for a more in-depth comparison). We remark that our result is optimal both in the analyticity assumption and in the fact that $\gamma$ might be (arbitrarily) greater than $0$, see \cref{ss:example}. The restriction $2\leq n\leq 7$ is due to the fact that minimizers are smooth only in these dimensions.

\subsection{Idea of the proof of \cref{thm:generic-isop} and \cref{thm:quantit_analytic}} 

The key idea for \cref{thm:generic-isop} and \cref{thm:quantit_analytic} is that the quantitative inequality \eqref{e:quant} with $\gamma=0$ corresponds to integrability of the minimizers, that is, roughly speaking, every null direction of the second variation can be killed by choosing a nearby minimizer. More precisely 

\begin{definition}
	We say that a minimizer $\Sigma$ of \eqref{e:isop} is \emph{integrable} if every Jacobi field on $\Sigma$ with $0$ average is the infinitesimal generator of a one parameter family of minimizers.
\end{definition} 
For example, in the case of the Euclidean space, balls are known to be the unique minimizers, and the $0$-average part of the kernel of the Jacobi operator is composed only of infinitesimal generators of translations, that is balls are integrable and the second order expansion gives the inequality. Since in our case integrability is in general false (see \cref{ss:example}) we have to use a stronger tool, that is the following infinite dimensional version of the so-called \L ojasiewicz inequality. 
\begin{lemma}[Quantitative inequality and \L ojasiewicz inequality]\label{l:quand_Loj}
	For any $n\geq 2$, let $(M^{n},g)$ be an analytic, compact Riemannian manifold and $\Sigma \subset M$ a smooth isoperimetric region of volume $|\Sigma|_{g} = V_{0}$. There exist constants $\delta,\gamma,C_0>0$, depending only on $(M,g)$ and $\Sigma$, such that if $E\subset M$ has $\|\chi_E-\chi_\Sigma\|_{L^1} \leq \delta$ and $|E|_{g} = V_{0}$ then
	\begin{equation}\label{e:quant_pointwise}
	\cP^{g}(E) - \cP^{g}(\Sigma) \geq C_0 \,\left(\alpha_{\delta}(E,g)\right)^{2+\gamma}
	\end{equation}
	where 
	\begin{equation}
	\alpha_\delta(E,g):=\inf\left\{ |E\Delta \tilde \Sigma|_g  \,:\,\tilde\Sigma \in \cM_{V_{0}},\, \|\chi_{\tilde\Sigma}-\chi_\Sigma\|_{L^1} \leq \delta \right\}\,.
	\end{equation}
	and we recall that $\cM^{g}_{V_{0}}$ are the isoperimetric regions of volume $V_{0}$.

	If $\Sigma$ is {integrable}, then we can take $\gamma=0$. If $\Sigma$ is strictly stable then we can replace \eqref{e:quant_pointwise} with the stronger
		\begin{equation}\label{e:quant_stable}
		\cP^{g}(E) - \cP^{g}(\Sigma) \geq C_0 \,\left|E\Delta\Sigma\right|^2.
		\end{equation}
\end{lemma}

This is a local version of \cref{thm:quantit_analytic} valid in every dimension as long as $\Sigma$ is smooth, and the proof of \cref{thm:quantit_analytic} follows from \cref{l:quand_Loj} and a simple compactness argument. On the other hand the proof of \cref{l:quand_Loj} is a consequence of the so-called \emph{selection principle}, introduced in \cite{CiLe} for the quantitative inequality in Euclidean space, and an infinite dimensional version of \L ojasiewicz inequality for competitors $E$ which are graphical on $\Sigma$, which replaces the so-called Fuglede inequality.

\medskip

\cref{thm:generic-isop} follows combining \eqref{e:quant_stable} with a bumpiness type theorem that guarantees that for generic metric and values of the enclosed volume $V_0$, minimizers of \eqref{e:isop} are strictly stable, that is the kernel of the second variation is empty. The only additional difficulty with respect to the results in \cite{White_b1,White_b2} is the parameter $V_0$, which corresponds essentially to a Lagrange multiplier.

\subsection{Technical discussion of related work}\label{s:previous} As mentioned above, there has been a lot of recent work on quantitative stability not just for the isoperimetric inequality but also for many other geometric (e.g. Brunn-Minkowski \cite{FiJe}), spectral (e.g.\ Faber-Krahn \cite{BrDeVe}), and functional (e.g.\ Sobolev \cite{Neumayer}) inequalities.  We defer to the recent survey of Fusco \cite{FuSurvey} for a more comprehensive list. When the underlying space and the extremizers are highly symmetric these results are often proven by symmetrization or rearrangement (see, e.g. \cite{CiFuMaPr}). In this vein we'd also like to point out the works \cite{ChIl, HySe}, which do not use symmetrization techniques but do exploit the richness of the symmetry group of the underlying space.

In the anisotropic setting, optimal transport techniques have been used with great success (see, e.g. \cite{BaKr}). However, usually convexity of the extremizers is required (e.g. to guarantee the necessary regularity on the transport map). Other techniques, such as the selection principle, often require understanding the spectrum of the relevant energy linearized around the extremizers (to obtain estimates like Fuglede's, \cite{Fug}). In the generality we consider here, there is very little one can say about the structure of the extremizers (i.e. isoperimetric regions) or symmetry of the underlying space. This lack of knowledge is our primary technical obstacle.

As alluded to above, we are able to overcome this obstacle by establishing the \L ojasiewicz-Simon type inequality \eqref{e:quant_pointwise}. \L ojasiewicz's work \cite{Loj}, was first applied to geometric analysis by Simon in \cite{Simon0}, in order to prove the regularity of solutions to certain elliptic PDE near isolated singularities. These ideas have been further developed by a number of different authors in a number of different settings, e.g., to understand the long term behavior of some gradient flows \cite{To, CoSpVe} or to prove results in the same vein as \cite{Simon0}, but either in the parabolic, see e.g.\ \cite{CoMi, ChSch}, or purely variational, see e.g.\ \cite{EnSpVe}, settings. See the introduction of \cite{Feehan1} and the references therein for a more comprehensive history. As far as we are aware, this is the first instance of a \L ojasiewicz-Simon type inequality being used to prove a quantitative stability result. 

\subsection{Results for stable minimal surfaces} 

We briefly note that the techniques used to prove \cref{l:quand_Loj} can be used to prove the following quantitative minimality result for minimal surfaces, related to the works \cite{Wh2,InMa}.
\begin{theorem}\label{th:min-surf-stab}
Consider a real analytic Riemannian manifold $(M^{n},g)$. Assume that $\Gamma^{n-1}\subset M$ is a smooth stable minimal hypersurface. Then, there is $\delta,C>0$ and $\gamma\geq 0$ (depending on $\Gamma,M,g)$ so that for $\cM$ the set of $\tilde\Gamma$ homologous to $\Gamma$ with the same mass and small flat norm $\F(\Gamma,\tilde\Gamma) < \delta$,\footnote{Here $\F(\Gamma,\Gamma') = \inf\{\bM(A) + \bM(B): A + \partial B = \Gamma - \Gamma'\}$ is the flat norm. See e.g., \cite[Section 2]{InMa}} we have that for any current $S$, homologous to $\Gamma$ with $\F(\Gamma,S)< \delta$, we have
\[
\bM(S) - \bM(\Gamma) \geq C \left( \inf_{\tilde\Gamma \in \cM} \F(S,\tilde\Gamma)\right)^{2+\gamma}. 
\]
where $\bM(\cdot)$ is the mass (area) of the current. 
\end{theorem}

This follows in a nearly identical manner to \cref{l:quand_Loj}. We note that with appropriate modifications, one can prove a similar result in higher co-dimension. It would be interesting to understand an analogue of \cref{th:min-surf-stab} for finite index surfaces (see \cite{Wh2}). 

\subsection{Plan of the paper} The first section, \cref{ss:prel}, is dedicated to fixing some notations and introducing some preliminary tools, particularly the various Banach manifolds we will use in the rest of the paper. In \cref{ss:Loj} we prove the infinite dimensional version of \L ojasiewicz inequality \cref{l:quand_Loj} and \cref{thm:quantit_analytic}, while \cref{ss:example} is dedicated to its optimality. Finally in \cref{ss:generic} we prove \cref{thm:generic-isop} and the bumpy metric result needed to do that.

\subsection{Acknowledgment}

O.C.\ was partially supported by an NSF grant DMS-1811059. He is grateful to Michael Eichmair for pointing out reference \cite{Ch:CMCiso}.  M.E.\ was partially supported by an NSF postdoctoral fellowship, NSF DMS 1703306 and by David Jerison's grant DMS 1500771. L.S.\ was partially supported by an NSF grant DMS-1810645. All three authors would like to thank Bozhidar Velichkov for many helpful ideas and conversations, without which this article would be much poorer. % B.V.\ was partially supported by Agence Nationale de la Recherche (ANR) by the projects GeoSpec (LabEx PERSYVAL-Lab, ANR-11-LABX-0025-01)
%and CoMeDiC (ANR-15-CE40-0006).

\section{Preliminaries and notations}\label{ss:prel}
We start by introducing some concepts that will be used throughout the paper. 
\subsection{The isoperimetric problem}

Recall that the distributional perimeter of $E\subset M$ is defined by
\[
\cP^g(E) = \sup \left\{\int_E \mathrm{div}_g(\phi)\, d{\rm vol}_g\mid \phi \in C^1(M; TM), \,\|\phi\|_{L^\infty} \leq 1\right\}.
\]
Sometimes, when it is clear in context, we will eliminate the dependence on $g$ from the notation. Then, for a fixed constant $0<V_{0} <|M^n|_g$, where $|\cdot|_g$ denotes the volume on $M$ induced by $g$, we study the minimization problem
\begin{equation}\label{e:isop}
\cI^{g}(V_{0}) : = \inf \{\cP^{g}(E) : E \in \cA^{g}_{V_{0}}\}
\end{equation}
where
\[
\cA^{g}_{V_{0}} : = \{E \subset M \, : \, \chi_{E} \in\mathrm{BV}(M), \, |E|_{g} = V_{0}\}
\]
is the set of Caccioppoli sets with volume $V$. If $\Omega \in \cA^{g}_{V_{0}}$ attains $\cI^{g}(V_{0})$, we say that $\Omega$ is isoperimetric. We let $\cM_{V_{0}}^{g}$ denote the set of isoperimetric regions of volume $V_{0}$. 

\subsection{Graphical regions} Let $\Sigma\subset M^n$ be such that $\de \Sigma$ is smooth and embedded and let $\nu_\Sigma$ be the normal to $\de \Sigma$ in $M^n$ pointing outside $\Sigma$. Let $f\colon \de\Sigma \to \R$, then the graph of $f$ is defined by 
$$
\graph(f):=\{ (x, \exp_x(f(x)\,\nu_\Sigma(x) ))\,:\,x\in \de \Sigma\}\,,
$$ 
and we will sometimes use the notation $\graph (f)=\de \Sigma+f$. Moreover we associate to each such graph a set of finite perimeter $\Sigma+f$ in such a way that $\de (\Sigma+f)=\de\Sigma+f=\graph(f)$, with orientation chosen so that $\nu_{E(f)}\cdot \nu_\Sigma\geq 0$.  When the set $\Sigma$ is clear from context, we will often abuse notation and use $f$ to refer to both the function but also the submanifold $\partial \Sigma + f$ or the subset $\Sigma +f$. 

If $u\colon N^{n-1}\to \Sigma^n$ is a smooth embedding from a compact orientable manifold $N^{n-1}$ to $M^n$, we will denote by $[u]$ the set of all maps of the form $u\circ \phi$, where $\phi \colon N\to N$ is a smooth diffeomorphism; that is the elements of $[u]$ are all parametrizations of the same surface $u(N)$.

\subsection{Banach manifolds} 
We will denote by 
$$
B_r^{k,\alpha}(h_0):=\{h\in C^{k,\alpha}\,:\,\|h-h_0\|_{C^{k,\alpha}}<r \}\,.
$$
Given $r, V_0>0$, and $\Sigma$ a minimizer of \eqref{e:isop} for a $C^3$ metric $g_0$ with $|\Sigma|_{g_0}=V_0$, we are interested in the following sets
\begin{gather}
\cB_r(\Sigma):=\{  f\in B_r^{2,\alpha} \,:\, |\Sigma+f|_{g_0}=|\Sigma|_{g_0} \} \label{e:Banach1}\\
\cB_r(\Sigma,g_0):=  \{ (f,g)\in B_r^{2,\alpha}\times B_r^{3}(g_0)\,:\, |\Sigma+f|_g=|\Sigma|_{g_0}    \}\label{e:Banach2}\\
\cB_r(\Sigma,g_0,V_0):=  \{ (f,g,V)\in B_r^{2,\alpha}\times B_r^{3}(g_0)\times B_r(V_0)\,:\, |\Sigma+f|_g=V    \}\label{e:Banach3}
\end{gather}

It is straightforward to see that these are Banach manifolds, we sketch the proof for the reader's convenience.

\begin{lemma}\label{l:volume_constrained_manifolds}
	Let $\Sigma$ be a smooth minimizer of the isoperimetric problem \eqref{e:isop} for the metric $g_0$. There exists $\delta>0$, depending on $\Sigma,g_0$, such that $\cB_\delta(\Sigma)$, $\cB_\delta(\Sigma,g_0)$ and $\cB_\delta(\Sigma,g_0,V_0)$ are separable, codimension one Banach submanifolds of the separable Banach spaces $C^{2,\alpha}$, $C^{2,\alpha}\times \Gamma$ and $C^{2,\alpha}\times \Gamma\times \R$ respectively (modeled on the Banach space of functions with zero average on $\de \Sigma$ with respect to the metric $g_0$).
\end{lemma}

\begin{proof}
	We sketch only the case $\cB_r(\Sigma)$, as the other two are the same. Separability follows from the separability of $C^{2,\alpha}$, so we only need to show that the function $F(f):=|\Sigma+f|_{g_0}-|\Sigma|_{g_0}$ is a submersion near $0$. To do this we observe that, by a well known computation (see for instance \cite[Lemma 3.1 and Section 7]{White_b2}) 
	$$
	DF(0)[v]=\int_{\de\Sigma} v\,d\sigma_{g_0}\,,
	$$
	where $d\sigma_{g_0}$ is the volume form of $\de \Sigma$ in the metric $g_0$. 
	Choosing $v$ as a constant, we immediately see that the differential is surjective, so that there exists $\delta>0$ depending on $\Sigma$ such that $\cB_\delta(\Sigma,g_0)$ is a Banach submanifold of $C^{2,\alpha}(\de \Sigma)$. Since the kernel of $DF(0)$ is the space of functions $v\in C^{2,\alpha}$ such that $\int_{\de \sigma} v\,d\sigma_{g_0}=0$, the proof is complete.
\end{proof}

In the sequel we will denote with $\nabla_\cB, \nabla^2_\cB$ the gradient and the Hessian respectively in $\cB_r(\Sigma)$, and with $\nabla_u,\nabla_g$ and so on the directional derivatives. Using this, we can define the submanifolds
\begin{gather}
\cM_r(\Sigma,g_0):=  \{ (f,g)\in \cB_r(\Sigma,g_0) \,:\, \nabla_u\cP^g(\Sigma+f)=0 \}\label{e:Banach4}\\
\cM_r(\Sigma,g_0,V_0):= \{ (f,g,V)\in \cB_r(\Sigma,g_0,V_0)\,:\, \nabla_u\cP^g(\Sigma+f)=0\}\,,\label{e:Banach5}
\end{gather}
which will be used in the proof of \cref{thm:generic-isop}.

\subsection{Properties of isoperimetric regions} 

The following result concerning regularity of isoperimetric regions is well known. See, e.g. \cite{Mag}.
\begin{theorem}\label{t:regularity}
We can choose representatives of minimizers of \eqref{e:isop} so that their boundaries are compact, have constant mean curvature, and are regular away from singular set of Hausdorff dimension at most $n-8$. 
\end{theorem}

Finally we recall the following 

\begin{lemma}\label{l:finite_components}
Let $\Sigma$ be an isoperimetric region in a closed Riemannian manifold $(M^n,g)$. There exists a number $L\in \N$, depending on $(M,g)$, such that the number of compact connected components of $\de \Sigma$ is bounded by $L$.
\end{lemma}

\begin{proof}
By \cite[Theorem 2.2]{MoJo} there is $\delta > 0$ so that if $|\Sigma|_{g} \in (0,\delta] \cup [|M|_{g}-\delta,|M|_{g})$ then $\partial\Sigma$ is connected (and indeed a perturbation of a coordinate sphere). Now, by \cref{lem:bdHiso} if $|\Omega|_{g}\in(\delta,|M|_{g}-\delta)$, then $\partial\Omega$ has constant mean curvature $|H| \leq C = C(M,g,\delta)$. By the boundedness of $H$, the monotonicity formula applied to each component of $\partial\Omega$ implies that $\cP^{g}(\Omega) \geq c L$ for some constant $c=c(M,g,\delta)>0$. However, it is easy to see that $\cI^{g}(V) \leq I_{0}=I_{0}(M,g)$ for all $V$, by e.g., foliating $(M,g)$ by the level sets of a Morse function. This completes the proof. 
\end{proof}

This will be used in the proof of \cref{l:LvsF} (to prove that the kernel of an elliptic operator over $\de \Sigma$ has finite dimension) and again in the proof of \cref{thm:generic-isop} (to conclude that there are only countably many diffeomorphism types for minimizers $\Sigma$ of \eqref{e:isop}).

\section{Proof of \cref{l:quand_Loj} and \cref{thm:quantit_analytic}}\label{ss:Loj}

The proof is divided in two parts. First, using a modification of the argument in Simon's \cite{Simon0}, based on the Lyapunov-Schmidt reduction and \L ojasiewicz inequality for analytic function, we prove \cref{l:quand_Loj} for graphs close to a smooth minimizers of \eqref{e:isop}. This can be interpreted as a generalization of Fuglede's inequality to the non-integrable case. In the second part we combine this result with a modification of the selection principle inspired by \cite{CiLe} to conclude the proof of \cref{l:quand_Loj}. 

Throughout this section, $(M,g)$ will be fixed, so we will not make explicit the dependence on $g$. 

%Since we don't care about the dependence on the metric in this section, we will drop it from all notations, so for example $\alpha(E)=\alpha(E,g)$ and $\delta\cP(E)=\delta\cP(E,g)$.

\subsection{Lyapunov-Schmidt reduction, integrability and strict stability} We start by recalling the following technical result whose proof is given in \cref{app:lyapunov_schmidt}. We denote by $K:=\ker(\nabla^2_\cB \cP(\Sigma))$. Notice that $\nabla^2_\cB \cP$ is the restriction of $J_\Sigma$, the Jacobi operator of $\Sigma$, to $T_0\cB(\Sigma)$, that is to functions with zero averages, and since $\de \Sigma$ is compact (by \cref{t:regularity}) and $J_\Sigma$ elliptic, $\dim K:=l<\infty$.

\begin{lemma}[Lyapunov-Schmidt reduction]\label{l:LS}
	Suppose $(M,g)$ is a $C^3$ manifold and $\Sigma$ is a smooth\footnote{By smooth here, we mean that the singular set is empty; then $\partial\Sigma$ will be as regular as $g$ allows.} minimizer of \eqref{e:isop}. There exists a neighborhood $U$ of $0$ in $T_0\cB(\Sigma)$ and a map $\Upsilon: K \cap U \rightarrow K^\perp \cap \cB(\Sigma)$, as regular as $g$, where the orthogonal complement is taken with respect to the $L^2$-inner product, such that 
	\begin{equation}\label{e:upsilonatzero}
	\Upsilon(0) = 0 \qquad\mbox{and}\qquad \nabla \Upsilon(0) = 0,
	\end{equation}
	and, in addition, 
	\begin{equation}\label{e:LSorth}
	\begin{cases}
	\pi_{K^\perp}(\nabla_{\cB(\Sigma)} \cP(\Sigma+\zeta+\Upsilon(\zeta)))=0&\forall \zeta \in K \cap U\\
	\pi_{K}(\nabla_{\cB(\Sigma)} \cP(\Sigma+\zeta+ \Upsilon(\zeta))) = \nabla P(\zeta)& \forall \zeta \in K \cap U,
	\end{cases}
	\end{equation}
	where $P\colon \R^l\to \R$ is the function defined by
	$$
	P(\zeta) = \cP(\Sigma+\zeta + \Upsilon(\zeta)) \quad \mbox{for every $\zeta \in K \cap U$}
	$$
	 and we identify $\zeta$ with the $l$-vector given by its coordinates on an orthonormal bases of the kernel $K$. Moreover we let $\mathcal L$ be the $l$-dimensional family defined by
	$$
	\mathcal L := \{\Sigma+\zeta + \Upsilon(\zeta)\mid \zeta  \in U \cap K\}\,.
	$$
	
	Now {assume that $g$ is analytic}. Then $P$ is analytic and satisfies the so-called \L ojasiewicz inequality at $0$ (see \cite[Corollary 4]{Feehan}),: there are constants $C,\delta>0, \gamma \geq 0$, depending on $\Sigma$, such that if $|\xi|<\delta$, then
	\begin{equation}\label{e:loj_p}
	P(\xi)-P(0)\geq C \left(\inf_{\{ \xi_0: \nabla P(\xi_0)=0\} } |\xi-\xi_0| \right)^{2+\gamma}\,.
	\end{equation} 
%	where the left hand side doesn't need an absolute value since $0$ is a minimizer of $P$. 
% the set $\cM\cap W$ is the collection of minimizers of \eqref{e:isop} which are graphical and close to $\Sigma$,
	For $W = \cB_{\delta}(\Sigma)$, we have
	\begin{equation}\label{e:critical_set}
	\cM \cap W:= \{\Sigma+\zeta + \Upsilon(\zeta)\mid \zeta \in U \cap K \quad \mbox{and} \quad \nabla P(\zeta) = 0\},
	\end{equation}
	and 
	\begin{equation}\label{e:p=cost}
	\tilde{\Sigma} \in \cM \cap W\qquad \mbox{implies}\qquad \cP(\tilde{\Sigma})=\cP(\Sigma)\,.  
	\end{equation}
	Moreover, for all $\zeta, \eta \in U\cap K$, there is a constant $C < \infty$, such that
	\begin{equation}\label{e:est_upsilon}
	\|\nabla\Upsilon(\zeta)[\eta]\|_{C^{2,\alpha}} \leq C\|\eta\|_{C^{0,\alpha}}.
	\end{equation}
	Finally, there exists a constant $C>0$ such that, if we denote with $u_\cL:=P_Ku+\Upsilon (P_K u)$, then the following key estimate holds
	\begin{equation}\label{e:stabby}
	\cP(\Sigma+u)-\cP(\Sigma+u_\cL)\geq C\, \left\| u- u_\cL \right\|^2_{W^{1,2}} \qquad\forall u\in \cB_{\delta}(\Sigma).
	\end{equation}
\end{lemma}

 \begin{definition}[Integrability and Strict Stability]\label{d:int}
	We say that a minimizer $\Sigma$ of \eqref{e:isop} is \emph{integrable} if every Jacobi field $u\in K=\ker(\nabla^2_\cB\cP(\Sigma))$ is the infinitesimal generator of a one parameter family of critical points of $\cP$ in $\cB_{\delta}(\Sigma)$; that is, if for every element $u\in K$ there exists a $1$-parameter family of diffeomorphisms $(\phi_t)_{t\in(-1,1)}$ such that $\phi_0 = Id$, $\frac{d}{dt}\phi_t=u(\phi_t)$, and
	$$
	(\phi_t)_\sharp(\Sigma)\in \cB_{\delta}(\Sigma)\quad \mbox{ is a critical point of $\cP$ in $\cB_{\delta}(\Sigma)$ for every $t\in(-1,1)$}\,.
	$$
	
	We say that  a minimizer $\Sigma$ of \eqref{e:isop} is \emph{strictly stable} if there exists $C>0$, depending on $\Sigma$, such that 
	\begin{equation}\label{e:strict_stab}
	\nabla^2_\cB\cP(\Sigma)[v,v] \geq C\, \|v\|^2_{W^{1,2}(\Sigma)} \qquad\mbox{for every $v\in T_0\cB_{\delta}(\Sigma)$}\,.
	\end{equation}
\end{definition}

In this case we can refine the Lyapunov-Schmidt decomposition to obtain the following lemma.

\begin{lemma}[Lyapunov-Schmidt and integrability]\label{l:LS+int}
	Under the same assumptions of \cref{l:LS}, the following holds.
	\begin{itemize}
		\item[(i)] If $g$ is analytic, then $\Sigma$ is integrable if and only if the function $P$ of \cref{l:LS+int} is constant. In particular if $\Sigma$ is integrable, then
		$$
		\cM \cap W\equiv \{\Sigma+\zeta + \Upsilon(\zeta)\mid \zeta  \in U \cap K\}\,.
		$$
		and moreover
		\begin{equation}\label{e:int}
		\cP(\Sigma+u)-\cP(\Sigma)\geq C\, \left\| u- u_\cL \right\|^2_{W^{1,2}} \qquad\forall u\in \cB_{\delta}(\Sigma).
		\end{equation}
		\item[(ii)] If $\Sigma$ is strictly stable and the metric $g\in C^3$, then $\mathcal L=\{\Sigma\}$ and moreover
		\begin{equation}\label{e:strict}
		\cP(\Sigma+u)-\cP(\Sigma)\geq C\, \left\| u \right\|^2_{W^{1,2}} \qquad\forall u\in \cB_{\delta}(\Sigma).
		\end{equation}
	\end{itemize} 
\end{lemma}

The proof of this fact is also contained in \cref{app:lyapunov_schmidt}.

\subsection{\L ojasiewicz inequality as a generalization of Fuglede's inequality}In this subsection we prove the main estimate of the paper.

\begin{lemma}[\L ojasiewicz meets Fuglede]\label{l:LvsF} Let $\Sigma$ be a smooth embedded orientable minimizer of \eqref{e:isop} on a manifold $(M,g)$. The following conclusions hold.
	
	If $g$ is analytic, then there exist constants $\delta(\Sigma), C(\Sigma),\gamma(\Sigma)>0$, all depending on $\Sigma,M,g$, such that
	\begin{equation}\label{e:Loj}
	\cP(\Sigma+u)-\cP(\Sigma)\geq C(\Sigma) \left(\alpha_{\delta(\Sigma)}(\Sigma+u)\right)^{2+\gamma(\Sigma)} \qquad \forall u \in \cB_\delta(\Sigma)\,.
	\end{equation} 
	
	If $g$ is analytic and $\Sigma$ is integrable, then we can take $\gamma\equiv 0$ in the above estimate, that is
	\begin{equation}\label{e:use_Loj}
	\cP(\Sigma+u)-\cP(\Sigma)\geq  C(\Sigma)\, \left( \alpha_{\delta(\Sigma)}(\Sigma+u) \right)^{2}\,\qquad \forall u \in \cB_\delta(\Sigma)\,.
	\end{equation}
	
	If $g\in C^3$ and $\Sigma$ is strictly stable, then the following estimate holds
	\begin{equation}\label{e:use_ss}
	\cP(\Sigma+u)-\cP(\Sigma)\geq  C(\Sigma)\, \left|(\Sigma+u)\Delta \Sigma\right|^2\,\qquad \forall u \in \cB_\delta(\Sigma)\,.
	\end{equation}
\end{lemma}

\begin{proof} We start by observing that
	\begin{equation}\label{e:W12_measure}
	\inf_{\tilde{u}\in \cM\cap \cB_\delta(\Sigma) }\|u- \tilde u\|_{W^{1,2}}\geq C\inf_{\tilde{u}\in \cM\cap \cB_\delta(\Sigma) }\|u- \tilde u\|_{L^1} \geq C \alpha_\delta(\Sigma+u)
	\end{equation}
	where the first inequality is Poincar\'e inequality and the second follows from the fact that $u,\tilde{\Sigma} \in \cB_{\delta}(\Sigma)$ implies that $u$ has small $C^{1,\alpha}$ norm when reparametrized over $\tilde{\Sigma}$. 	

	Now we prove \eqref{e:Loj}. Let $u\in \cB(\Sigma)$ and $u_\cL\in \cL$ be as in \cref{l:LS} and write
	\begin{equation}\label{e:step1}
	\cP(\Sigma+u)-\cP(\Sigma)= \underbrace{\cP(\Sigma+u)-\cP(\Sigma+u_\cL)}_{=:I^\perp}+\underbrace{\cP(\Sigma+ u_\cL)-\cP(\Sigma)}_{=:I_\cL}\,,
	\end{equation}
	For the first term we simply use \eqref{e:stabby}, therefore to conclude we only need to estimate $I_\cL$. We distinguish three cases.
	
	\medskip
	
	\noindent \textbf{$\Sigma$ is strictly stable.} In this case $u_\cL\equiv 0$ and \eqref{e:use_ss} follows immediately from  \eqref{e:strict} and \eqref{e:W12_measure}.
	
	\medskip
	
	\noindent \textbf{$\Sigma$ is integrable.} In this case we have by \eqref{e:p=cost} that $I_\cL=0$, therefore we have by \eqref{e:stabby} and \eqref{e:step1}, and the fact that $u_\cL\in \cL\cap W=\cM\cap W$, that
	\begin{align*}
	\cP(\Sigma+u)-\cP(\Sigma)\geq C\, \|u-u_\cL\|_{W^{1,2}}^2 \geq C \left(\inf_{\bar{u}\in \cM\cap \cB_\delta(\Sigma) }\|u-\bar{u}\|_{W^{1,2}}\right)^{2}\,.
	\end{align*}
	Combined with \eqref{e:W12_measure}, this proves \eqref{e:use_Loj}.
	
	\medskip
	
	\noindent \textbf{$\Sigma$ is not integrable.} We identify $u_\cL$ with $\xi \in \R^l$, via $u_\cL=\xi+\Upsilon \xi$ where $\xi$ is the projection of $u$ onto the kernel of $\nabla^2_\cB$. Using the definition of $P$ and \eqref{e:loj_p} we get
	\begin{align}\label{e:parallel}
	I_\cL&=P(u_\cL)-P(0)\geq C\,\left(\inf_{\{ \xi_0:\nabla P(\xi_0)=0\} } |\xi-\xi_0| \right)^{2+\gamma}\\\notag
		&\geq C\, \left(\inf_{\tilde u\in \cM\cap \cB_\delta(\Sigma) }\|u_\cL-\tilde u\|_{W^{1,2}}\right)^{2+\gamma}\,,
	\end{align}
	where in the last inequality we used standard estimates on elements of the kernel of the linear elliptic operator $\nabla^2_\cB$ and \eqref{e:est_upsilon}. To conclude, we combine the inequalities \eqref{e:stabby} and \eqref{e:parallel}, with the simple fact that $a^{2+\gamma}+b^{2+\gamma}\geq C(\gamma) (a+b)^{2+\gamma} $, for every $a,b>0$, to conclude that
	\begin{align*}
	\cP(\Sigma+u)-\cP(\Sigma)
	&\geq  C\, \|u-u_\cL\|_{W^{1,2}}^2 +C\, \left(\inf_{\tilde u\in \cM\cap \cB_\delta(\Sigma) }\|u_\cL-\tilde u\|_{W^{1,2}}\right)^{2+\gamma}\\
	&\geq C\, \left( \|u-u_\cL\|_{W^{1,2}} + \inf_{\tilde u\in \cM\cap \cB_\delta(\Sigma) }\|u_\cL- \tilde u\|_{W^{1,2}} \right)^{2+\gamma}\\
	&\geq C \left(\inf_{\tilde u\in \cM\cap \cB_\delta(\Sigma) }\|u-\tilde u\|_{W^{1,2}}\right)^{2+\gamma} 
	\end{align*}
	which, together with \eqref{e:W12_measure}, concludes the proof of the proposition.
\end{proof}

\subsection{Proof of \cref{l:quand_Loj}} 
Let $\Sigma, V_0$ be as in \cref{l:quand_Loj}. Let $\delta(\Sigma),C(\Sigma)>0$ and $\gamma:=\gamma(\Sigma) \geq 0$ be the constants given by \cref{l:LvsF}, depending on $\Sigma$. 

Given a set of finite perimeter $E \subset \cA_{V_0}\cap W_\delta$, where
\begin{equation}
W_\delta:=\{ F\subset M\,:\, \chi_F\in \mathrm{BV}(M),\,\|\chi_F-\chi_\Sigma\|_{L^1}\leq \delta  \}\,,
\end{equation} 
we can define the associated ``energy" relative to $\Sigma$ 
\begin{equation}\label{e:QforE} 
\cQ(E,\gamma):= \inf \left\{\liminf_k \frac{\delta \cP(F_k)}{\alpha_\delta(F_k)^{2+\gamma}} \mid \{F_k\}_k \subset \cA_{V_0},\: \alpha_\delta(F_k) > 0,\: |F_k \Delta E| \rightarrow 0\right\}
\end{equation} 
where
\[
\delta\cP(F_{k}) = \cP(F_{k}) - \cI(V_{0})
\]
is the isoperimetric defect.

With $\gamma > 0$ fixed as above,  assume that there is a sequence of ``bad" sets $E_k \in \cA_{V_0}\cap W_\delta$ such that 
\[
\delta \cP(E_k) \leq \frac{1}{k}\alpha_\delta(E_k)^{2+\gamma}.
\]

The trivial bound of $\alpha_\delta(E_k) \leq 2V_0$ implies that $\delta\cP(E_k) \rightarrow 0$. Note, by compactness in the space of functions of bounded variation and the lower-semicontinuity of perimeter, passing to a subsequence, we can guarantee that $E_k \rightarrow \tilde{\Sigma} \in \cM\cap W_\delta$ in the sense of sets of finite perimeter and, therefore, $\alpha_\delta(E_k) \rightarrow 0$ as well. We have just shown that the (local) quantitative isoperimetric inequality is equivalent to the statement that \begin{equation}\label{e:lowerboundforsigma} 
\inf_{\tilde{\Sigma} \in \cM\cap W_\delta} \cQ(\tilde{\Sigma},\gamma) > 0.
\end{equation}

In order to prove \eqref{e:lowerboundforsigma} we are going to use the following version of the selection principle of \cite{CiLe}. 

\begin{proposition}[Selection Principle]\label{p:sel_pri} Assume that $\cQ(\Sigma,\gamma) < \infty$. There exists a sequence of sets of finite perimeter $E_k \subset M$ with the following properties 
	\begin{itemize}
		\item[(i)] $\alpha_\delta(E_k)>0$ as $k\to \infty$;
		\item[(ii)] $\cQ(E_k,\gamma)\to \inf_{\tilde{\Sigma} \in \cM\cap W_\delta} \cQ(\tilde{\Sigma},\gamma) $ as $k\to \infty$;
		\item[(iii)] there exists a smooth $\Sigma_0 \in \cM\cap W_\delta$ such that 
		$$
		\inf_{\tilde{\Sigma} \in\cM\cap W_\delta} \cQ(\tilde \Sigma,\gamma) = \cQ(\Sigma_0, \gamma)
		$$ 
		and functions $u_k\in C^{1,\alpha}(\partial \Sigma_0)$ such that $E_k:=\Sigma_0+u_k$ and $\|u_k\|_{C^{1,\alpha}}\to 0$ as $k\to \infty$.
	\end{itemize}
\end{proposition}

The proof of \cref{p:sel_pri} is given in \cref{ss:selection_principle} and is a modification of the one in \cite{CiLe} with the simplification that the ambient space is compact and the complication that once again we do not know the shape of the minimizers nor the growth of the isoperimetric profile $V\mapsto \cI(V)$. Notice that one of the reason for this local version is the choice of $\delta$ so that $\de \Sigma_0$ is smooth, since it is sufficiently close to $\Sigma$.

We are now ready to conclude the proof of \cref{l:quand_Loj}. If $\cQ(\Sigma,\gamma) = \infty$, then it follows from \cref{l:LvsF} (and the triangle inequality) that \cref{l:quand_Loj} holds. Otherwise, we can apply \cref{p:sel_pri}: since $\Sigma_0\in W_\delta$ is a minimizer of \eqref{e:isop} and $\partial\Sigma$ is assumed to be smooth, choosing $\delta$ sufficiently small depending on $\delta(\Sigma)$, $\eps$-regularity guarantees that $E_k = \Sigma + \tilde u_{k}$ for some $\tilde u_k \in \cB_{\delta(\Sigma)}(\Sigma)$.

Then by (ii) \cref{p:sel_pri}  we have
\begin{align}
\inf_{\tilde{\Sigma} \in \cM\cap W_\delta} \cQ(\tilde{\Sigma},\gamma)
	&=\lim_{k\to \infty} \cQ(E_k,\gamma)= \lim_{k\to \infty}  \frac{\delta \cP(\Sigma+\tilde u_k)}{\alpha_{\delta(\Sigma)}(\Sigma+\tilde u_k)^{2+\gamma}}\notag
	\stackrel{\eqref{e:Loj}}{\geq} C(\Sigma)>0
\end{align}
where the second equality follows from the fact that $\alpha_\delta(E_k) > 0$ for every $k$ (i.e., (i) of \cref{p:sel_pri}). This implies \eqref{e:lowerboundforsigma} and thus concludes the proof.
\qed

\subsection{Proof of \cref{thm:quantit_analytic}}

Let $\delta(\Sigma), C(\Sigma)>0, \gamma(\Sigma)\geq 0$ be the constants of \cref{l:quand_Loj} for each $\Sigma \in \cM$. Consider the covering $\left(B_{\sfrac{\delta(\Sigma)}2}(\Sigma)\right)_{\Sigma\in \cM}$ of $\cM$ with respect to the $L^1$-norm, and recall that $\cM$ is compact, so that there exists a finite subcover $\left(B_{\sfrac{\delta(\Sigma_j)}2}(\Sigma_j)\right)_{j=1}^J$ of $\cM$. Set
\begin{equation}\label{e:choice_constants}
\delta_0:=\min_{j=1,\dots,J}\delta(\Sigma_j)\,,
\qquad 
\gamma_0:=\max_{j=1,\dots,J}\gamma(\Sigma_j)
\qquad\mbox{and}\qquad
C_0:=\min_{j=1,\dots,J}C(\Sigma_j)\,.
\end{equation}

We claim that there exists $\alpha_0>0$ such that 
\begin{equation}\label{e:small_regime}
\alpha(E)<\alpha_0\quad \mbox{implies}\quad \delta \cP(E)\geq C_0 \,\left(\alpha(E)\right)^{2+\gamma_0}\,.
\end{equation} 
Indeed suppose not, than there exists a sequence $E_j$ of sets of finite perimeter such that $\alpha(E_j)\to 0$ such that
$$
\delta \cP(E_j)< C_0 \, \left(\alpha(E_j)\right)^{2+\gamma_0}\,.
$$
By standard compactness argument, up to a subsequence $E_j\to \bar \Sigma\in \cM$, so that there exists $N$ sufficiently large satisfying
$$
\|\chi_{E_N}-\chi_{\bar \Sigma}\|_{L^1}\leq \frac{\delta_0}2\,.
$$
Then, by triangle inequality and definition of $(\Sigma_j)_{j=1}^J$, we can assume without loss of generality that 
$$
\|\chi_{E_N}-\chi_{\Sigma_1}\|_{L^1}\leq \delta(\Sigma_1),
$$
and that, by our contradiction assumption and the definitions of $C_0,\gamma_0$,
$$
\delta \cP(E_N)< C(\Sigma_1) \, \left(\alpha(E_N)\right)^{2+\gamma(\Sigma_1)} \leq C(\Sigma_1) \, \left(\alpha_{\Sigma_1}(E_N,\delta)\right)^{2+\gamma(\Sigma_1)}  \,.
$$
This is a contradiction with \cref{l:quand_Loj}.

Next suppose $\alpha_0\leq \alpha(E)\leq 2 |M|_g$, then we recall the following fact (whose proof is a simple contradiction argument combined with the fact that $M$ is compact):

\begin{lemma}[{\cite[Lemma 3.1]{CiLe}}]
	For every $\alpha_0 >0$ there exists $\delta_0 >0$ such that, for any $E$, if $\delta \cP(E)< \delta_0$, then $\alpha(E)<\alpha_0$.
\end{lemma}

Then we have that in our regime $\cP(E)-\cP(\Sigma)\geq \delta_0$, and so
\begin{equation}\label{e:large_regime}
\delta \cP(E)\geq \delta_0\geq \frac{\delta_0}{(2 |M|_g)^{2+\gamma_0}} (\alpha(E))^{2+\gamma_0}\,.
\end{equation}
Choosing $C_1:=\min\left\{C_0,\frac{\delta_0}{(2 |M|_g)^{2+\gamma_0}} \right\}$, \cref{thm:quantit_analytic} follows from \eqref{e:small_regime} and \eqref{e:large_regime}.
\qed

\section{Optimality of \cref{thm:quantit_analytic}}\label{ss:example}
In this section we prove \cref{thm:example}, giving an example demonstrating the sharpness of \cref{thm:generic-isop} and \cref{thm:quantit_analytic}. Finally we discuss the possibility of extending our results to the case of non-compact, finite volume manifolds.

We begin by proving the following relatively standard result. See e.g.\ \cite{PedRit:99,Ped,Sou} for more refined  statements.
\begin{lemma}\label{lem:S1Sn-iso}
There is $R_{0} = R_{0}(n)$ so that for $R \geq R_{0}$, if we consider the product metric $g_{R}$ on $\mathbb{S}^{1}(R) \times \mathbb{S}^{n-1}(1)$, then every isoperimetric region $\Omega \subset M$ with volume $|\Omega| = \frac 12 |\mathbb{S}^{1}(R) \times \mathbb{S}^{n-1}(1)|$ is of the form
\[
\Omega = (t_{0},t_{0}+\pi R) \times \mathbb{S}^{n-1} 
\]
for $t_{0} \in \R$.  
\end{lemma}
\begin{proof}
For $n=2$ this can easily be proven by passing to the universal cover $\R^{2}$ and using the classification of embedded constant curvature curves. We thus consider $n\geq 3$. The proof we give below holds for $3\leq n \leq 7$, but can be easily modified to accommodate for a singular set in higher dimensions. 

Take $R_{k}\to\infty$ and consider a sequence of isoperimetric regions 
\[
\Omega_{k}\subset \mathbb{S}^{1}(R_{k}) \times \mathbb{S}^{n-1}
\]
with $|\Omega_{k}| = \frac 12 |\mathbb{S}^{1}(R_{k}) \times \mathbb{S}^{n-1}| \to \infty$. By comparison with the expected minimizer, we have that
\[
\cP(\Omega_{k}) \leq 2 |\mathbb{S}^{n-1}|. 
\]
Moreover, because the Ricci curvature of $\mathbb{S}^{1}(R)\times \mathbb{S}^{n-1}$ is non-negative and vanishes only in the $\mathbb{S}^{1}(R)$ directions, we see\footnote{This is a standard argument: if there are two boundary components, take a function in the second variation that is $1$ on one component and $-\lambda$ on another, where $\lambda$ is chosen so that the function integrates to zero. Non-negativity of the Ricci curvature implies that $|A|^{2} + \textrm{Ric}(\nu,\nu)$ vanishes identically along each component.} that the reduced boundary of $\Omega_{k}$ has exactly one component unless $\Omega_{k}$ is the form asserted in the lemma. 

We now claim that the mean curvature $H_{k}$ of $\partial^{*}\Omega_{k}$ remains uniformly bounded as $k\to\infty$. This follows exactly as in \cref{lem:bdHiso} since all of the metrics $g_{R}$ are locally isometric. Thus, using the monotonicity formula, we find that each component of $\partial^{*}\Omega_{k}$ has uniformly bounded (extrinsic) diameter, say by $T_{0}$. From this, the proof easily follows, because if $\partial^{*}\Omega_{k}$ has only one component, then $\partial^{*}\Omega_{k}\subset [t_{k},t_{k}+T_{0}]\times \mathbb{S}^{n-1}$ for some $t_{ k}$. This implies that either $|\Omega_{k}| = O(1)$ or $|\Omega_{k}| \geq \frac 34 |\mathbb{S}^{1}(R_k)\times \mathbb{S}^{n-1}|$ for $k$ sufficiently large.  This is a contradiction. 
\end{proof}

We now prove \cref{thm:example}. We begin with the non-analytic case (the third assertion in the Theorem) and explain how the proof can be modified for the analytic case at the end of the section. Consider a fixed $R\geq R_{0}$ for $R_{0}$ from the previous lemma. Consider a sequence of smooth functions $\varphi_{k} : \R \to (1/2,2)$ so that 
\begin{enumerate}
\item $\varphi_{k}$ is $2\pi R$-periodic,
\item $\varphi_{k}(r) = 1$ for $|r-1| > \frac 1 k$
\item $\varphi_{k}$ converges smoothly to $1$ as $k\to\infty$,
\item $\varphi_{k}(1) = 1-\frac 1 k$ is the unique minimum of $\varphi_{k}$, and
\item $\varphi_{k}$ is strictly decreasing on $(1/2,1)$ and strictly increasing on $(1,3/2)$. 
\end{enumerate}
A simple calibration argument (based on (5) above) combined with \cref{lem:S1Sn-iso} shows that for $k$ sufficiently large, the regions $(1,R_{k})\times \mathbb{S}^{n-1}$ and their complement are the unique isoperimetric regions of half the volume in the warped product metric $g_{k}$ on $\mathbb{S}^{1}(R) \times \mathbb{S}^{n-1}$ given by
\[
g_{k} = dr^{2} + \varphi_{k}(r)^{2} g_{\mathbb{S}^{n-1}}.
\] 
(Above, $R_{k} = 1 + \pi R + o(1)$ as $k\to\infty$). We fix such a $k$ for the remainder of this section.

We now consider the sets $\Gamma_{\delta} : = (1+\delta,\rho_{\delta})\times \mathbb{S}^{n-1}$ where $\rho_{\delta}$ is chosen so that the volume of $\Gamma_\delta$ is equal to $\frac 12 |M|$ for all $\delta>0$ small. Note that
\[
\rho_{\delta} =  \rho_{0} + \int_{1}^{ 1+\delta} \varphi(r)^{n-1} dr
\]
Thus,
\[
|\Gamma_{\delta}\Delta\Gamma_{0}| = 2 |\mathbb{S}^{n-1}| \int_{1}^{1+ \delta} \varphi(r)^{n-1} dr \geq c \delta
\]
On the other hand, for some $C_n > 0$,
\[
\cP(\Gamma_{\delta}) - \cP(\Gamma_{0}) =C_n \left(\varphi(1+\delta)^{n-1} - \varphi(1)^{n-1}\right), 
\]
(recall that $\varphi \equiv 1$ outside of a small neighborhood of $1$).

Now, suppose that $\varphi(1+r) - \varphi(1)$ vanishes faster than any polynomial. Then, we see that 
\[
\cP(\Gamma_{\delta}) - \cP(\Gamma_{0}) \leq C_{j}\delta^{j}
\]
for any $j > 0$. This shows that it cannot be true that 
\[
\cP(\Gamma_{\delta}) - \cP(\Gamma_{0}) \geq C|\Gamma_{\delta}\Delta\Gamma_{0}|^{2+\gamma}
\]
for any $C, \gamma>0$, independent of $\delta$. 

To show that for a general analytic metric, it is necessary to allow $\gamma>0$ (arbitrarily large) is slightly more involved. We sketch the modifications here. Choose an analytic warping function $\varphi$ that is $\pi R$ periodic, with unique minima $\varphi(1) < 1$ at $1$ (and hence $1+\pi R$), so that the warping function is strictly decreasing on $(0,1)$ and strictly increasing on $(1,2)$, and so that $\varphi>1$ outside of $(0,2)$. Assuming $\varphi(1+x) = \varphi(1) + x^{2m}$ for  small $x$ and a large positive integer $m$, shows that one cannot take $\gamma=0$ (and that $\gamma > 0$ can be arbitrarily large). \qed

\subsection{Manifolds with metric of finite volume} We briefly comment on the situation for $(M,g)$ non-compact but still with finite volume. By \cite{Mor:GMT}, isoperimetric regions exist for all volumes $V_{0}\in (0,|M|_{g})$. However, it seems possible that such an $(M,g)$ exists where some isoperimetric region has infinitely many components (compare to \cref{l:finite_components}). While each of these components might satisfy a \L ojasiewicz inequality, it seems plausible that the associated constants, $\gamma$, are unbounded, in which case one could construct a counterexample to the finite-volume analogue of \cref{thm:generic-isop} or \cref{thm:quantit_analytic}. It would be interesting to rigorously construct such an example.

\section{Proof of \cref{thm:generic-isop}}\label{ss:generic}

We first quickly adapt some of the results in \cite{White_b2} to our setting and prove \cref{thm:generic-isop}.

\subsection{More Banach manifolds}

Given $\Sigma$, a minimizer of \eqref{e:isop} for the metric $g_0$ and with volume $V_0$, we consider the space of critical points for the isoperimetric problem near $(0,g_0,V_0)$ defined by
\begin{gather*}
\cM_r(\Sigma,g_0,V_0):=\left\{(f,g,V)\in \cB_r(\Sigma,g_0,V_0)\,:\, \nabla_u\cP^g(\Sigma+f)=0\right\}\,,
\end{gather*}
where $0<r<\delta$ as in \cref{l:volume_constrained_manifolds}.

\begin{proposition}\label{p:local_critical_points}
	There exists $0<\delta_1<\delta$, depending on $\Sigma,g_0,V_0$, such that $\cM_{\delta_1}(\Sigma, g_0,V_0)$ is a separable, smooth Banach submanifold of $\cB_{\delta}(\Sigma,g_0,V_0)$ and such that the projection 
	$$
	\Pi\colon \cM_{\delta_1}(\Sigma, g_0,V_0) \to \Gamma\times \R \quad \mbox{is a Fredholm operator of index }0\,.
	$$
\end{proposition}

\begin{proof} Since the statement is local we can apply \cite[Theorem 1.2]{White_b2} with (using his notation), $X=\cB_{\delta}(\Sigma,g_0,V_0)$, $Y=C^{0,\alpha}(\de \Sigma)$, $\Gamma=\Gamma\times V$ and $H(u,g)=\nabla_u \cP^g(\Sigma+f)$. It is well known that the operator $\nabla_{uu}\cP^{g_0}(\Sigma)$ is a self-adjoint Fredholm map of order $0$ (it is the restriction of the usual Jacobi operator of $\de \Sigma$ in the metric $g_0$ to variations with zero average, that is, to the tangent of $\cB$). Moreover for every nonzero $v\in \ker(\nabla_{uu}\cP^{g_0}(\Sigma))$, let $g(s)\in \Gamma$ be the one parameter family of metrics defined in a neighborhood of $\de \Sigma$ by
	$$
	g(s)(z):=(1+sf(z))\,g_0(z)
	$$
	where $f(z)=0$ for every $z\in \de \Sigma$. Then, following the computation in \cite[Theorem 2.1]{White_b2} we can find $f$ such that
	$$
    \frac{\de^2}{\de s\de t}\Big|_{t=0=s} \cP^{g(s)}(\Sigma+tv)\neq 0\,.
	$$ 
	Since for $t,s$ sufficiently small, we can always choose $V(s,t):=|\Sigma+tv|_{g(s)}$ such that $(\Sigma+tv,g(s),V(t,s))\in \cB_{\delta}$, condition (C) of \cite[Theorem 1.2]{White_b2} is satisfied and \cref{p:local_critical_points} is proved. 
\end{proof}

Finally, by a standard procedure (see \cite[Theorem 2.1]{White_b2}), we can patch together all the local neighborhoods $\cM(\Sigma, g_0, V_0)$ to obtain a Banach manifold containing all the critical points for \eqref{e:isop} for varying metric and values of the volume, but fixed diffeomorphism type (as we are only working with local parametrizations). 

\begin{proposition}\label{p:final_manifold}
	Let $N^{n-1}$ and $M^n$ be smooth compact  manifolds and let $\Gamma$ be the collection of $C^3$ riemannian metrics on $M^n$. Let $[u]$ denote the class of $C^{2,\alpha}$ embeddings $u\colon N\to M$ up to diffeomorphism, i.e. $v\in [u]$ if and only if $v=u\circ \phi$, with $\phi \colon N\to N$ a smooth diffeomorphism. Let 
	\begin{multline*}
	\cM(N):=\{ ([u],g,V)\,:\, \mbox{$u(N)$ is the boundary of a critical point of \eqref{e:isop}} \\ \mbox{w.r.t. $g$ and  $V$}\}\,,
	\end{multline*}
	Then $\cM(N)$ is a smooth separable Banach manifold and the map 
	$$
	 \cM(N) \ni ([u],g,V)\mapsto \Pi([u],g,V):=(g,v)\in \Gamma\times V
	$$
	is a Fredholm operator of index $0$ and the kernel of $D\Pi([u],g,V)$ has dimension equal to the nullity of the second variation of $u(N)$ with respect to the metric $g$ in linear space of functions with zero average on $u(N)$.
\end{proposition}

\begin{proof}
	As observed in the preliminaries, given $u,v$ embeddings of $N$ in $M$ such that $\|u-v\|_{C^{2,\alpha}}\ll1$, we can find a function $f\in C^{2,\alpha}(u(N))$ such that $v(N)=u(N)+f$, and viceversa; if $f\in C^{2,\alpha}(\partial \Sigma=u(N))$  has small norm, then we can find $v\in C^{2,\alpha}(N,M)$,  $\|u-v\|_{C^{2,\alpha}}<<1$, such that $\partial \Sigma+f=v(N)$.  
	
	With this identification in mind we can use  \cref{p:local_critical_points}  to find local charts for $\cM$. The rest of the proposition follows exactly as in \cite[Theorem 2.1]{White_b2}.
\end{proof}

\subsection{Proof of \cref{thm:generic-isop}} First of all notice that for every diffeomorphism type $N^{n-1}$ we can apply Sard--Smale \cite[Theorem 1.3]{SaSm} to $\cM(N)$ and $\Pi$ to show that for every fixed $N$ there is an open and dense subset $\mathcal G_N\subset \Gamma\times \R$ such that every minimizer $u(N)$ of \eqref{e:isop} with $(g,v)\in \mathcal G_N$ is non-degenerate, that is, strictly stable. Since by \cref{l:finite_components}, every minimizer of \eqref{e:isop} has finitely many connected compact components and  since there are countably many diffeomorphism types for compact manifolds $(N_i)_{i\in \N}$, we can consider the open dense subset $\mathcal G:=\bigcap_{i\in \N}\mathcal G_{N_i}$ of $\Gamma\times \R$. Its projection $\mathcal G$ and $U$ on $\Gamma$ and $\R$ respectively are also open and dense, since the projection is an open map.

Now let $g\in \mathcal G $ and $V\in (0,|M|_g)\cap U$. Since every smooth minimizer of \eqref{e:isop} is strictly stable, we have that there are only finitely many minimizers with volume $V$ in the metric $g$. Now the result follows using \eqref{e:quant_stable} and letting $C(g)$ be the minimum of the constants in \eqref{e:quant_stable}.
\qed

\subsection{Some further consequences} We notice that as an outcome of the previous theorem we also have the following result valid in every dimension

\begin{corollary}
	For an open and dense set of metrics and volumes, smooth minimizers of \eqref{e:isop} are strictly stable and thus satisfy \eqref{e:quant_stable} with a constant $C$ depending only on $(M,g)$ and $V_{0}$. This is the generic analogue of \cref{l:quand_Loj}.
\end{corollary}

Finally if one uses $\cB_r(\Sigma,g_0)$ instead of $\cB_r(\Sigma, g_0,V_0)$ and argues as in the previous two subsections it is easy to conclude the following

\begin{corollary}\label{coro:fixV-perturbg}
	Let $V\in \R$. There exists an open and dense set of metrics, $\mathcal G\subset \Gamma$, such that for every $g\in \mathcal G$  there exists a constant $C(g,V)>0$ such that if $\Sigma\in \cA_{V}^g$ is a minimizer of \eqref{e:isop} in $(M,g)$, then
	\begin{equation}
	\delta \cP^g(E) \geq C(g,V) \,\left|E\Delta \Sigma\right|_g^{2}\\,\qquad \mbox{for every $E\in \cA^g_{V_0}$.}
	\end{equation}
\end{corollary}

\appendix

\section{Proof of \cref{l:LS,l:LS+int}}\label{app:lyapunov_schmidt}

In this section we prove the Lyapunov-Schmidt reduction and its version in the integrable case 

\begin{proof}[Proof of Lemma \ref{l:LS}]
	Recall that $K:=\ker \nabla_\cB^2 \cP(\Sigma)\subset T_0\cB(\Sigma)$ and define the operator
	$$
	\cN(\zeta) := P_{K^\perp} \nabla_\cB \cP(\Sigma+\zeta) + P_K \zeta\,,
	$$ 
	where $P_K, P_{K^\perp}$ denote the projections on $K,K^\perp$ with respect to the inner product of $L^2(\Sigma)$.
	Since $\Sigma$ is a critical point for $\cP$ restricted to $\cB(\Sigma)$, we have that $\mathcal N(0) = 0$. Furthermore, 
	$$
	\nabla \mathcal N(0)[\zeta] = \frac{d}{dt} \mathcal N(t\zeta)|_{t= 0} = P_{K^\perp} \nabla^2_\cB \cP(\Sigma)[\zeta, -] + P_K \zeta.
	$$ 
	In particular $\nabla \mathcal N(0)$ has trivial kernel. We observe that $\nabla^2_\cB \cP(\Sigma)=J_\Sigma$ restricted to $T_0\cB(\Sigma)$, that is to smooth functions with zero average, therefore we can apply Schauder estimates to obtain that $\nabla \mathcal N(0)$ is an isomorphism (in a neighborhood of zero) from $C^{2,\alpha}(\de \Sigma)\cap T_0\cB(\Sigma)$ to $C^{0,\alpha}(\de \Sigma)\cap T_0\cB(\Sigma)$. 
	
	We apply the inverse function theorem to $\mathcal N$ in this neighborhood, producing the map $\Psi := \mathcal N^{-1}$ which is a bijection from a neighborhood of $0$, $W \subset C^{0,\alpha}(\de \Sigma)\cap T_0\cB(\Sigma)$ to $U$, a neighborhood of $0$ in $C^{2,\alpha}(\de \Sigma)\cap \cB(\Sigma)$. We claim that our desired map is given by 
	$$
	\Upsilon := P_{K^\perp}\circ \Psi: K \rightarrow K^\perp\cap \cB(\Sigma)\,.
	$$ 
	In particular, for $\zeta \in K$ we have $\Psi(\zeta) = \zeta + \Upsilon(\zeta)$. The first conclusion of \eqref{e:upsilonatzero} is trivial as $\Upsilon(0) = \Upsilon(\mathcal N(0)) = P_{K^\perp}(\Psi(\mathcal N(0))) = 0$.
	
	To check \eqref{e:LSorth}, we first notice that 
	\begin{equation}\label{e:zetaintermsofpsi}
	\zeta = \mathcal N(\Psi(\zeta)) = P_{K^\perp}\nabla \cP(\Sigma+\Psi(\zeta)) + P_K\Psi(\zeta).
	\end{equation} 
	Applying $P_K$ or $P_{K^\perp}$ to both sides of this equation we get 
	$$
	P_{K} \zeta = P_K \Psi(\zeta) \qquad\mbox{and}\qquad P_{K^\perp} \zeta =  P_{K^\perp}\delta \cP(\Sigma+\Psi(\zeta)).
	$$ 
	Plugging the first identity into the second we obtain 
	$$P_{K^\perp} \zeta = P_{K^\perp} \nabla_\cB\cP(\Sigma+P_K \zeta + \Upsilon(\zeta)),
	$$ 
	which implies, for $\zeta \in K\cap U$, that 
	$$0 = P_{K^\perp} \nabla_\cB \cP(\Sigma+\zeta + \Upsilon(\zeta)).
	$$ 
	
	To prove the second line of \eqref{e:LSorth}, we compute, for any $\eta \in K$; 
	$$
	\begin{aligned} 
	\left\langle \nabla F(\zeta),\eta\right\rangle =
	& \nabla_\cB \cP(\Sigma+\zeta + \Upsilon(\zeta))[\eta+ \nabla \Upsilon(\zeta)[\eta]]\\
	=& \nabla_\cA \cP(\Sigma+\zeta + \Upsilon(\zeta))[\eta],
	\end{aligned}
	$$ 
	which implies the second claim of \eqref{e:LSorth} (as $\eta \in K$ is arbitrary). The second inequality above follows from the fact that $\nabla \Upsilon(\zeta)[\eta] \in K^\perp$ (as the image of $\Upsilon$ is in $K^\perp$) and then from the first line of \eqref{e:LSorth}.
	
	To prove  \eqref{e:critical_set} we turn to \eqref{e:zetaintermsofpsi}. Let $\Sigma+\eta$ be an arbitrary critical point of $\cP$ in a neighborhood of zero. We write $\eta = \Psi(\zeta)$, and \eqref{e:zetaintermsofpsi} reads $\zeta = P_K \eta$. This implies 
	$$
	\eta = P_K \eta + P_{K^{\perp}} \eta = \zeta + P_{K^\perp} \Psi(\zeta) = \zeta + \Upsilon(\zeta),
	$$ 
	as desired (the condition on $\nabla F$ follows trivially from \eqref{e:LSorth}).  To show \eqref{e:p=cost} we recall that, by the gradient version of the \L ojasiewicz inequality, there exist $\gamma_0,C_0,\delta_0$, depending on $\Sigma$ such that
	$$
	|P(\xi)-P(0)|^{1-\gamma}\leq C_0\, |\nabla P(\xi)| \qquad \mbox{for every }\xi\in B_{\delta_0}\,,
	$$ 
	from which $P(\xi)=P(0)$ as long as $\xi$ is a critical point of $P$, as desired.
	 
	To prove \eqref{e:est_upsilon} we write 
\begin{align*}
	\eta = \nabla_\cB \mathcal N(\Psi(\zeta))[\nabla_\cA \Psi(\zeta)[\eta]] & = P_K \nabla_\cB \Psi(\zeta)[\eta] \\ & \qquad + P_{K^\perp}\nabla^2_\cB \cP(\Sigma+\Psi(\zeta))[\nabla_\cB\Psi(\zeta)[\eta], -],
\end{align*}
	which implies 
	\begin{align*}
	&P_K \nabla_\cB \Psi(\zeta)[\eta] + P_{K^\perp} \nabla^2_\cB \cP(\Sigma)[\nabla_\cB \Psi(\zeta)[\eta],-] \\
	&= \eta + P_{K^\perp}\left(\nabla^2_\cB \cP(\Sigma) - \nabla^2_\cB \cP(\Sigma+\Psi(\zeta))\right)[\nabla_\cB \Psi(\zeta)[\eta], -].
	\end{align*}
	Applying $P_K$ to both sides of the above equation we get $P_K \nabla_\cB \Psi(\zeta)[\eta] = P_K \eta = \eta$. Applying $P_{K^\perp}$ to both sides and taking $C^{0,\alpha}$ norms, yields  
	$$
	\begin{aligned} 
	\|\nabla_\cB\Upsilon(\zeta)[\eta]\|_{C^{1,\alpha}} \leq
	& \|P_{K^\perp} \nabla^2_\cB \cP(\Sigma)[\nabla_\cB \Psi(\zeta)[\eta],-]\|_{C^{0,\alpha}}\\ 
	\leq& \|P_{K^\perp}\left(\nabla^2_\cB \cP(\Sigma) - \nabla_\cB^2 \cP(\Sigma+\Psi(\zeta))\right)[\nabla_\cB \Psi(\zeta)[\eta], -]\|_{C^{0,\alpha}}\\ 
	\leq& \varepsilon \|\nabla_\cB\Psi(\zeta)[\eta]\|_{C^{1,\alpha}},
	\end{aligned}
	$$ 
	where $\varepsilon > 0$ is a constant which can be taken arbitrarily small with the size of the neighborhood $U$. Note the first inequality above follows from Schauder estimates on the operator $\nabla^2_\cB\cP(\Sigma)$. 
	Writing $\nabla_\cB \Psi(\zeta)[\eta] = \nabla \Upsilon(\zeta)[\eta] + P_K \nabla_\cB\Psi(\zeta)[\eta]$ we have 
	$$ 
	\|\nabla \Upsilon(\zeta)[\eta]\|_{C^{2,\alpha}} \leq C \|P_K \nabla_\cB \Psi(\zeta)[\eta]\|_{C^{1,\alpha}} \simeq \|P_K \nabla_\cB \Psi(\zeta)[\eta]\|_{C^{0,\alpha}},
	$$ 
	as $P_K$ is a finite dimensional projection (so all norms are equivalent). Recalling the above observation, that $P_K \nabla_\cB \Psi(\zeta)[\eta] = \eta$, finishes the proof.
	
	Finally, to prove \eqref{e:stabby} we notice that, since $\Sigma$ is a minimizer of $\cP$ in $\cB(\Sigma)$, there exists a constant $C$, depending on $\Sigma$, such that 
	\begin{equation}\label{e:stab}
	\nabla^2_\cB \cP(\Sigma)[\eta,\eta]\geq C\,\|\eta\|_{W^{1,2}}^2 \qquad \forall \eta \in K^\perp\,.
	\end{equation}
	Then we can use a simple Taylor expansion to deduce, with the notation $u^\perp:= u-u_\cL=u-P_K u-\Upsilon(P_ku)\in K^\perp$, that
	\begin{align}\label{e:perp}
	\cP&(\Sigma+u)-\cP(\Sigma +u_\cL)\\
	&=\nabla_\cB \cP(\Sigma+u_\cL)[u^\perp]+\nabla^2_\cB \cP(\Sigma+u_\cL)[u^\perp,u^\perp] + o\left(\|u^\perp\|^2_{W^{1,2}}  \right)\notag\\
	&= \underbrace{\langle  \nabla_\cB \cP(\Sigma+u_\cL), u^\perp  \rangle_{L^2}}_{\stackrel{\eqref{e:LSorth}}{=}0}+\nabla_\cB^2 \cP(\Sigma+u_\cL)[u^\perp,u^\perp] + o\left(\|u^\perp\|^2_{W^{1,2}}  \right)\notag\\
	&=\nabla_\cB^2 \cP(\Sigma)[u^\perp,u^\perp]-\left(\nabla_\cB^2 \cP(\Sigma)[u^\perp,u^\perp]-\nabla^2_\cB \cP(\Sigma+u_\cL)[u^\perp,u^\perp] \right)\notag\\
	& \qquad + o\left(\|u^\perp\|^2_{W^{1,2}}  \right)\notag\\
	&\stackrel{\eqref{e:stab}}{\geq} C\, \|u^\perp\|_{W^{1,2+}}^2
	\end{align}
	where the last inequality follows by the continuity of $\nabla^2_\cB \cP$ at $\Sigma$ by choosing the norm of $u$, and so $W$, small enough, together with \eqref{e:LSorth}.
\end{proof}

Next we prove the integrable and strictly stable versions of the Lyapunov-Schmidt reduction, which are a simple modification of the argument above essentially already contained in \cite{AdSi}.

\begin{proof}[Proof of Lemma \ref{l:LS+int}]
	The integrability condition is equivalent to
	\begin{equation}\label{e:int2}
	\forall \phi\in K \quad \exists (\Psi_s)_{s\in(-1,1)}\subset C^2(\Sigma, \Sigma^\perp)\,\,\mbox{s.t. }
	\begin{cases}
	\lim_{s\to 0}\Psi_s=0\\
	\nabla_\cB\cP(\Psi_s)=0\quad \mbox{for }s\in(-1,1)\\
	\displaystyle\frac{d}{ds}\Big|_{s=0}\Psi_s=\lim_{s\to 0}\frac{\Psi_s}{s}=\phi \,.
	\end{cases}
	\end{equation}  
	Assume \eqref{e:int2} holds, and recall the definition $P(\mu) = \cP(\Sigma+\mu + \Upsilon(\mu))$. If $F \equiv 0$ in a neighborhood of zero then we are done. Otherwise we can write $P(\mu) = P_p(\mu) + P_R(\mu)$ where, $P_p \not\equiv 0$,  $P_p(\lambda \mu) = \lambda^p P(\mu/|\mu|)$ for $\lambda > 0$ and $P_R(\mu)$ is the sum of homogeneous polynomials of degrees $\geq p+1$ (here we use the analyticity of $P$). Note that there exists some $\phi \in K$ such that $\nabla F_p(\phi) \neq 0$. Let $\Psi_s$ be the one-parameter family of critical points that is generated by $\phi$ (as in \eqref{e:int2}). 
	
	\noindent As $\Psi_s$ is a critical point, Lemma \ref{l:LS} allows us to write $\Psi_s = \phi_s + \Upsilon(\phi_s)$ where $\phi_s \in K$ and $\frac{\phi_s}{s} \rightarrow \phi$ as $s\downarrow 0$. Computing 
	$$
	0 = \nabla_\cB \cP(\Psi_s) = \nabla F(\phi_s) = \nabla P_p(\phi_s) + \nabla P_R(\phi_s) = s^{p-1}\nabla P\left(\frac{\phi}{|\phi|}\right) + o(s^{p-1}).
	$$ 
	Divide the above by $s^{p-1}$ and let $s\downarrow 0$ to obtain a contradiction to $\nabla P_p(\phi) \neq 0$. 
	
	In the other direction assume that $F \equiv 0$ in a neighborhood of $0$. This implies that $\nabla F \equiv 0$ in a (perhaps slightly smaller) neighborhood of $0$. Therefore, for any $\mu \in K$, letting $\Psi_s = s\mu + \Upsilon(s\mu)$ and recalling \eqref{e:est_upsilon} establishes \eqref{e:int2}.
	
	Next, since we have proven that $P$ is constant on $\cC$, \eqref{e:int} follows immediately from \eqref{e:stabby} and the fact that $u_\cL \in \cC$.
	\medskip
	
	Finally, if $\Sigma$ is strictly stable, then $K=\{0\}$ which immediately implies $\cC=\{0\}$ and so $u_\cL=0$, which gives \eqref{e:strict}.
\end{proof} 

\section{Proof of \cref{p:sel_pri}}\label{ss:selection_principle}

The proof of \cref{p:sel_pri} is obtained combining results from \cite{CiLe,InMa,Wh2} and we will recall the fundamental steps over the following subsections, leaving many standard details to the reader. The basic idea is that the $E_k$ will be minimizers to a penalized version of energy in \eqref{e:QforE}, where the penalization guarantees that we recover $\inf_{\tilde{\Sigma} \in \cM\cap W_\delta} \cQ(\tilde{\Sigma},\gamma)$ in the limit. 

The existence of the $E_k$ and the fact that they satisfy properties (i) and (ii) of \cref{p:sel_pri} is covered in \cref{p:minofpenal}. The smooth convergence of property (iii) of \cref{p:sel_pri} is proven in \cref{l:smoothconv}. 

For simplicity of notation, in this section we will denote
$$
\alpha(E):=\alpha_\delta(E),
\qquad \cA:=\cA_{V_0},
\qquad W:=W_\delta.
$$
We emphasize that we are \emph{assuming} that $\cQ(\Sigma,\gamma) < \infty$ in this section. 
	
	Before starting the proof we observe the following simple facts.

\begin{lemma}[Properties of $\cQ(-,\gamma)$]
\label{r:propofQ}	
The energy $\cQ(-,\gamma)$ satisfy the following properties.
	\begin{itemize}
		\item If $\alpha_\delta(E) > 0, E \subset \cA$, then $\displaystyle{\cQ(E,\gamma) = \frac{\delta \cP(E)}{\alpha(E)^{2+\gamma}}}$.
		\item If $E_k \subset \cA$ and $E_k \stackrel{L^{1}}{\longrightarrow} E$, then $\cQ(E, \gamma) \leq \liminf_{k} \cQ(E_k,\gamma)$. This follows from the lower-semicontinuity of perimeter and a diagonal argument. 
	\end{itemize}
\end{lemma}

\subsection{The Penalized Minimization Problem} By the definition of $\cQ(\tilde{\Sigma},\gamma)$ and a diagonal argument, there exists $\{W_j\}_j \subset \cA$ such that \begin{equation}\label{e:Wjdef} \begin{aligned} 
\left| \cQ(W_j, \gamma)- \inf_{\tilde{\Sigma} \in \cM \cap W_{}} \cQ(\tilde{\Sigma},\gamma)\right| <& \frac{1}{j},\quad
0<\alpha(W_j) < 1, \quad
\alpha(W_j)  \rightarrow 0\,.
\end{aligned}
\end{equation}
We want to ``regularize" these $W_j$ and so we introduce the following penalized functionals
\begin{equation}\label{e:Qjdef}
\cQ_j(E,\gamma):=\cQ(E,\gamma)+\left(\frac{\alpha(E)}{\alpha(W_j)}-1\right)^{2}
\end{equation}
where $(W_j)_j$ is as in \eqref{e:Wjdef}. The content of the following proposition is that minimizers to  $\cQ_j(-, \gamma)$ exist and are also an approximating sequence for $\inf_{\tilde{\Sigma} \in \cM\cap W_\delta} \cQ(\tilde{\Sigma},\gamma)$ (i.e. they satisfy \eqref{e:Wjdef}).

\begin{proposition}[Minimizers of $\cQ_j$]\label{p:minofpenal}
	There exists sets of finite perimeter $\{E_j\}_j \subset \cA$ such that for each $j$, $\cQ_j(E_j, \gamma) \leq \cQ_j(S, \gamma)$ for all other sets $S \in \cA$. Furthermore, 
	$$
	\alpha(E_j) > 0,\quad \alpha(E_j) \rightarrow 0 \qquad \mbox{and} \qquad \left| \cQ(E_j, \gamma)- \inf_{\tilde{\Sigma} \in \cM\cap W_{}} \cQ(\tilde{\Sigma},\gamma)\right| \rightarrow 0.
	$$ 
	Finally, perhaps passing to a subsequence, $E_j  \stackrel{L^{1}}{\longrightarrow} \Sigma_0$ where $\Sigma_0 \in \cM\cap W$ is smooth and $\cQ(\Sigma_0, \gamma) = \inf_{\tilde{\Sigma} \in \cC\cap W} \cQ(\tilde{\Sigma},\gamma)$. 
\end{proposition}

\begin{proof}
	The existence of a minimizer follows from BV-compactness and the lower semi-continuity of the energy $\cQ_j(-,\gamma)$ (see \cref{r:propofQ}, second bullet point). 
	
	If $\alpha(E_j) = 0$ for any  $j > 1$, then $E_j \in \cM\cap W$ and we have that 
	\begin{align*}
	\inf_{\tilde{\Sigma} \in \cM\cap W} \cQ(\tilde{\Sigma}, \gamma) 
	&\leq  \cQ(E_j, \gamma) = \cQ_j(E_j,\gamma) - 1 \leq \cQ_j(W_j, \gamma) - 1\\
	&= \cQ(W_j, \gamma) - 1 \leq \inf_{\tilde{\Sigma} \in \cC\cap W} \cQ(\tilde{\Sigma}, \gamma) + \frac{1}{j} - 1,
	\end{align*}
	which is a contradiction as long as $j >1$. 
	
	A similar argument shows that $\alpha(E_j) \rightarrow 0$. Indeed for any subsequence $E_{j_k}$ we have 
	\begin{align*}
	\lim_k \left(\frac{\alpha(E_{j_k})}{\alpha(W_{j_k})} - 1\right)^{2} 
	&\leq \lim_k \cQ_{j_k}(E_{j_k},\gamma) \leq \lim_k \cQ_{j_k}(W_{j_k},\gamma) = \lim_k \cQ(W_{j_k},\gamma) \\
	&= \inf_{\tilde{\Sigma} \in \cC\cap W} \cQ(\tilde{\Sigma},\gamma)< \infty.
	\end{align*}
	Since $\alpha(W_j) \rightarrow 0$ it follows that $\alpha(E_j) \rightarrow 0$. 
	
	Of course, we can similarly argue that 
	$$
	\cQ(E_j, \gamma) \leq \cQ_j(E_j, \gamma) \leq \cQ_j(W_j, \gamma) = \cQ(W_j, \gamma) \leq \inf_{\tilde{\Sigma} \in \cC\cap W} \cQ(\tilde{\Sigma},\gamma) + 1 < \infty\,,
	$$ 
	where we emphasize that we have assumed that $\cQ(\Sigma,\gamma) < \infty$.
	
	This implies that $\delta\cP(E_j) \rightarrow 0$
so $E_j  \stackrel{L^{1}}{\longrightarrow} \Sigma_0$ for some $\Sigma_0 \in \cM\cap W$. Note that $\partial \Sigma_0$ is automatically smooth by the definition of $W$ and the assumption that $\partial \Sigma$ is smooth.
	
	We have proven that the $E_j$ (perhaps passing to a subsequence) satisfy the requirements of an approximating sequence in the definition of $\cQ(\Sigma_0, \gamma)$. Therefore, 
	$$
	\begin{aligned}  \cQ(\Sigma_0,\gamma) \leq& \lim_j \cQ(E_{j}, \gamma) \leq \lim_j \cQ_{j}(E_{j}, \gamma)\\ \leq& \lim_j \cQ_{j}(W_{j}, \gamma) = \lim_j \cQ(W_{j},\gamma) \\
	=& \inf_{\tilde{\Sigma} \in \cC\cap W} \cQ(\tilde{\Sigma},\gamma).\end{aligned}
	$$
	This implies that
	$$
	\lim_j \cQ(E_j, \gamma) = \inf_{\tilde{\Sigma} \in \cC} \cQ(\tilde{\Sigma},\gamma)=\cQ(\Sigma_0,\gamma)
	$$ 
	and, finally, that 
	 \begin{equation}\label{e:alpharatio} 
	 \lim_{j\to \infty}\frac{\alpha(E_j)}{\alpha(W_j)}=1,
	 \end{equation}
	 completing the proof. 
\end{proof}

\subsection{Almost-Minimizers and Smoothness for the $E_j$} In this subsection we will prove that the $E_j$ satisfy the hypothesis of \cref{p:sel_pri}. Note that we only have to verify the smooth convergence property (property (iii)) as the first two properties are guaranteed by \cref{p:minofpenal}.

We will prove this smooth convergence by first showing that the $E_j$'s are \emph{almost-minimizers} for perimeter with uniform constants. Then smooth convergence will follow from regularity theory for almost-minimizers and a standard argument in the calculus of variations (see the proof of \cref{l:smoothconv} below for more details).

Our first lemma is that $E_j$ minimizes perimeter in the class $\cA$ up to an error which is proportional to the area of the symmetric distance between $E_j$ and the competitor. It is important to note that the constant of proportionality is uniform over the index. 

\begin{lemma}\label{l:minsarequasimins}
	There exists a $\Lambda > 0$ and $j_0 \in \mathbb N$ such that for all $F\in \cA$ and all $j \geq j_0$ we have $$\cP(E_j) \leq \cP(F) + \Lambda |E_j \Delta F|.$$
\end{lemma}

\begin{proof}
	Without loss of generality we can assume that $\cP(F) \leq \cP(E_j)$. We also let $j_0$ be large enough such that 
	\begin{equation}\label{e:assonj0}
	\begin{aligned}
	 \alpha(E_j) &\leq \frac{1}{2}\\ 
	 |\alpha(E_j) - \alpha(W_j)| &\leq \frac{\alpha(W_j)}{2}\\
	\cQ(E_j, \gamma) &\leq  \inf_{\tilde{\Sigma} \cap W\in \cM} \cQ(\tilde{\Sigma},\gamma) + 1\,.
	\end{aligned}
	\end{equation}  
	Such a $j_0$ exists by \cref{p:minofpenal} and \eqref{e:alpharatio}. Next we distinguish two cases.
	
	\smallskip
	
	\noindent {\bf Case 1:} $\alpha(E_j)^{2+\gamma} \leq |E_j \Delta F|$. Since $\cQ(E_j,\gamma) \leq \inf_{\tilde{\Sigma} \in \cM\cap W} \cQ(\tilde{\Sigma},\gamma) + 1$ and $\alpha(E_j)>0$, we get 
	\begin{align}\label{e:case1reduction}
	\cP(E_j) &\leq \cI(V_{0}) +  \alpha(E_j)^{2+\gamma}\left(\inf_{\tilde{\Sigma} \in \cM\cap W} \cQ(\tilde{\Sigma},\gamma) + 1\right)\\ \notag
	&\leq  \cP(F)+ |E_j \Delta F|\left(\inf_{\tilde{\Sigma} \in \cM\cap W} \cQ(\tilde{\Sigma},\gamma) + 1\right)\\ \notag
	&\leq \cP(F)  + \Lambda |E_j \Delta F|,
	\end{align}
	completing the proof in this case. 
	
	\smallskip
	
	\noindent {\bf Case 2:}  $|E_j \Delta F| < \alpha(E_j)^{2+\gamma}$. We know the inequality $\cQ_j(E_j,\gamma) \leq \cQ_j(F,\gamma)$ which implies that \begin{align}\label{e:case2reduction}
	\cP(E_j) 
	&\leq \cP(F) + \underbrace{\delta\cP(F)\left(\frac{\alpha(E_j)^{2+\gamma}}{\alpha(F)^{2+\gamma}} - 1\right)}_{I} \\\notag
	&+ \underbrace{\alpha(E_j)^{2+\gamma} \left(\left(\frac{\alpha(F)}{\alpha(W_j)} - 1\right)^{2} - \left(\frac{\alpha(E_j)}{\alpha(W_j)} - 1\right)^{2}\right)}_{II}.
	\end{align}
	We can estimate $II$ in
	\eqref{e:case2reduction} as follows: 
	\begin{align}\label{e:estimateforII}
	II 
	&\leq \left(\frac{\alpha(E_j)}{\alpha(W_j)}\right)^2 \left(\alpha(F) + \alpha(E_j) - 2\alpha(W_j)\right)\left(\alpha(F) - \alpha(E_j)\right)\\
	& \leq C|\alpha(F) - \alpha(E_j)| \leq C|F \Delta E_j|,
	\end{align}
	where the second inequality follows from the estimates in \eqref{e:assonj0} and the last inequality follows from the triangle inequality. 
	In order to estimate $I$ we observe that, by assumption $\alpha(E_j) \leq 1/2$, so we have that 
	\begin{equation}\label{e:comp}
	|E_j \Delta F| \leq \alpha(E_j)^{2+\gamma} \leq \frac{1}{2}\alpha(E_j)\Rightarrow \frac{1}{2}\alpha(E_j) \leq \alpha(F) \leq 2 \alpha(E_j).
	\end{equation}
	It follows that
	\begin{align}\label{e:estimateforI}
	I &
	\leq \delta \cP(F)\frac{\alpha(E_j)^{2+\gamma} -\alpha(F)^{2+\gamma}}{\alpha(F)^{2+\gamma}} 
	\leq C \cQ(E_j,\gamma) \left(\alpha(E_j)^{2+\gamma} - \alpha(F)^{2+\gamma}\right)\notag\\
	&\leq C(\inf_{\tilde{\Sigma} \in \cM} \cQ(\tilde{\Sigma},\gamma) + 1)(\alpha(E_j) - \alpha(F)) \leq C |F \Delta E_j|,
	\end{align}
	where the second inequality follows from \eqref{e:comp} and the fact that $\cP(F)\leq \cP(E_j)$, while the third inequality comes from \eqref{e:assonj0} and the estimate $x^r - y^r \leq C(x-y)$ for $0 \leq y\leq x \leq 1$. 
	
	Putting \eqref{e:estimateforI} and \eqref{e:estimateforII} together with \eqref{e:case2reduction} finishes the proof of Case 2 and thus concludes the proof of the lemma. 
\end{proof}

In the following theorems we first prove that the $E_j$ are almost-minimizers (also known as $\Lambda$-minimizers) in the sense that $$\cP(E_j) \leq \cP(F) + Cr^{n}, \forall \chi_F= \chi_E\quad \text{in} \quad M\backslash B(x,r).$$ We invoke two results, those of \cite{Tamanini} and \cite{CiLe}, which are about almost and quasi-minimizers in $\R^n$. However, both theorems are local statements and any local statement about almost-minimizers in $\R^n$ also holds for almost-minimizers in a $C^3$-manifold (and vice versa). This can be seen by working local coordinates and freezing the metric. This introduces an error which is of an order comparable to the scale and thus does not change the almost-minimization property. 

Our first result is due to  Tamanini, \cite{Tamanini}, and states that sets which minimize (or minimize up to a lower order error) perimeter amongst sets of the same area are actually almost-area minimizers. It is important to note that the constants are independent of $j$. 

\begin{proposition}\label{p:minsarealmostmins}
	There exists a $C = C(M, \Sigma) > 0$, an $r_0 = r_0(M, \Sigma) > 0$, an $\alpha = \alpha(M) \in (0,1)$  and a $j_0 \in \mathbb N$ (which again depends on $\Sigma$) such that for all $j \geq j_0$, all $x \in U$and all $r < r_0$, if $\chi_F \in \mathrm{BV}(M)$ with $\chi_F = \chi_{E_j}$ on $M \backslash B(x, r)$, then \begin{equation}\label{e:almostmin} \cP(E_j) \leq \cP(F) + Cr^{n}.\end{equation}
\end{proposition}

\cref{p:sel_pri} will now follow from standard facts about the regularity of almost-minimizers and smooth convergence. We write the formal statement here and collect the salient facts in the proof. 

\begin{lemma}\label{l:smoothconv}
	The $E_k$ satisfy condition (iii) of \cref{p:sel_pri}, that is if $\Sigma_0 \in \cM\cap W$ is as in \cref{p:minofpenal}, then there are functions $u_k\in C^{1,\alpha}(\partial \Sigma_0)$ such that $E_k:=\Sigma_0+u_k$ and $\|u_k\|_{C^{1,\alpha}}\to 0$ as $k\to \infty$.
\end{lemma}

\begin{proof}
	For almost-minimizers, convergence in the $\mathrm{BV}$ sense implies convergence in the Hausdorff sense. Smooth convergence then follows from $\varepsilon$-regularity for almost-minimizers. These are standard facts about almost-minimizers, see, e.g. \cite[Propositions 2.1-2.2]{CiLe}.
\end{proof}

\section{Boundedness of mean curvature for isoperimetric regions} 

In this appendix we recall the uniform boundedness of the mean curvature of isoperimetric regions whose volume is not very small (or close to $|M|_{g})$. 
\begin{lemma}[\cite{MoRo,Ch:CMCiso}]\label{lem:bdHiso}
For $2\leq n\leq 7$, fix $\delta >0$ and $(M^{n},g)$ a closed Riemannian manifold with $C^{3}$-metric. There is $C=C(M,g,\delta) < \infty$ so that if $\Omega\in \cA^{g}_{V}$ is an isoperimetric region with $V = |\Omega|_{g} \in (\delta, |M|_{g}-\delta)$, then the mean curvature of $\partial\Omega$ satisfies $|H| \leq C$. 
\end{lemma}
\begin{proof}
Fix $(M,g)$ and $\delta>0$ and assume for contradiction that there are isoperimetric regions $\Omega_{j}\subset (M,g)$ with $|\Omega_{j}|_{g} \in (\delta,|M|_{g}-\delta)$ with mean curvature $H_{j}$ satisfying $\lambda_{j} : = |H_{j}| \to \infty$. 

Choosing $x_{j} \in \partial \Omega_{j}$, we can rescale by $\lambda_{j}$ around $x_{j}$ to find isoperimetric regions $\tilde\Omega_{j}$ in $(M,\tilde g_{j},x_{j})$ which converges in $C^{3}_{\textrm{loc}}$ to the flat metric on $\R^{n}$. Passing to a subsequence, $\tilde\Omega_{j}$ converge in the local Hausdorff sense to $\tilde\Omega$ a locally isoperimetric region in $\R^{n}$; moreover $\partial\tilde\Omega_{j}$ converge in $C^{2,\alpha}_{\textrm{loc}}$ to $\partial\tilde \Omega$. Hence, $\partial\tilde\Omega$ has constant mean curvature $\pm1$. On the other hand, $\partial\tilde \Omega$ is stable, in the sense that 
\[
\int_{\partial\tilde\Omega} |A|^{2} \varphi^{2} d\cH^{n-1} \leq \int_{\partial\tilde\Omega} |\nabla \varphi|^{2} d\cH^{n-1} 
\]
for any $\varphi \in C^{1}_{c}(\partial\Omega)$ with $\int_{\partial\tilde\Omega} \varphi d\cH^{n-1} = 0$. Because $|H| = 1$, we find $|A|^{2}\geq \frac 1 n$, so 
\[
\int_{\partial\tilde\Omega}  \varphi^{2} d\cH^{n-1} \leq n \int_{\partial\tilde\Omega} |\nabla \varphi|^{2} d\cH^{n-1} 
\]
for any $\varphi \in C^{1}_{c}(\partial\Omega)$ with $\int_{\partial\tilde\Omega} \varphi d\cH^{n-1} = 0$. 

Suppose that $\partial\tilde\Omega$ were compact for all choices of $x_{j} \in\partial\Omega_{j}$. Then, $\Omega_{j}$ would be close to a union of an increasing number of regions close to coordinate spheres. Using this, we would conclude that $\cP_{g}(\Omega_{j}) \to \infty$, a contradiction. As such, we will assume that $\partial\tilde\Omega$ is non-compact.

Standard arguments (cf.\ \cite{FiC}) imply that there is $R>0$ sufficiently large so that  
\[
\int_{\partial\tilde\Omega}  \varphi^{2} d\cH^{n-1} \leq n \int_{\partial\tilde\Omega} |\nabla \varphi|^{2} d\cH^{n-1} 
\]
holds for any $\varphi \in C^{1}_{c}(\partial\tilde\Omega\setminus B_{R})$ (i.e., $\partial\tilde\Omega$ is strongly stable outside of a compact set). 

Taking $\varphi = \psi^{\frac{n-1}{2}}$ for $\psi \in C^{1}_{c}(\partial\tilde\Omega\setminus B_{R})$ and using H\"older's inequality, we find
\[
\int_{\partial\tilde\Omega}  \psi^{n-1} d\cH^{n-1} \leq C \int_{\partial\tilde\Omega} |\nabla \psi|^{n-1} d\cH^{n-1} .
\]
Choose $\psi$ an ambient radial function that is $0$ for $|x| < R$ increases to $1$ for $|x| \in [R+1,\rho]$ and then cuts off to $0$ for $|x| > 2\rho$. We can arrange that $|\nabla \psi| \leq C \rho^{-1}$. Thus, we find that
\[
\cH^{n-1}(\partial\tilde\Omega \cap (B_{\rho} \setminus B_{R})) \leq C (1 +  \rho^{1-n} \cH^{n-1}(\partial\tilde\Omega \cap B_{2\rho}))
\]
Letting $\rho\to\infty$, we deduce a contradiction if we can show that $\cH^{n-1}(\partial\tilde\Omega \setminus B_{R}) = \infty$ and $\cH^{n-1}(\partial\tilde\Omega\cap B_{\rho}) \leq C \rho^{n-1}$. The first fact follows from the monotonicity formula (since $|H|=1$) applied to small balls. The second follows since $\tilde\Omega$ is locally isoperimetric in the sense that $\tilde\Omega'$ with $\tilde\Omega\Delta \tilde \Omega' \Subset B_{R}$ and $|\tilde\Omega \cap B_{R}| = |\tilde\Omega' \cap B_{R}|$ has $\cP(\partial\tilde\Omega' ; B_{R}) \geq \cP(\partial\tilde\Omega ; B_{R})$, allowing us to compare $\tilde\Omega$ to $(\tilde \Omega \setminus B_{\rho}) \cup B_{r(\rho)}$, where $r(\rho) \leq \rho$ is chosen to preserve the enclosed volume. This is a contradiction, completing the proof. 
\end{proof}

\bibliographystyle{plain}
\bibliography{references-Cal}

\end{document}